\documentclass[10pt]{siamltex}

\usepackage{amsmath,amsfonts,amscd,amssymb,bm,cite}
\usepackage{mathrsfs,listings,url}
\usepackage{graphics,color}
\usepackage{float}
\usepackage[titletoc,page]{appendix}
\usepackage{subfigure}
\usepackage[colorlinks=true, bookmarksopen,
            pdfauthor={CST},
            pdfcreator={pdftex},
            pdfsubject={algorithms},
            linkcolor={blue},
            anchorcolor={black},
            citecolor={firebrick},
            filecolor={magenta},
            menucolor={black},
            pagecolor={red},
            plainpages=false,pdfpagelabels,
            urlcolor={db} ]{hyperref}
\usepackage{comment}
\usepackage{verbatim}
\usepackage{marginnote}
\usepackage{float}
\usepackage[makeroom]{cancel}
\usepackage[pdftex]{graphicx}
\usepackage[section]{placeins}
\usepackage{mathptmx}
\usepackage{cases}
\usepackage{subeqnarray}
\usepackage{dsfont}
\usepackage{bm}
\usepackage[utf8]{inputenc}
\usepackage[T1]{fontenc}

\makeatletter
\@namedef{subjclassname@1991}{Subject}
\@namedef{subjclassname@2000}{Subject}
\@namedef{subjclassname@2010}{Subject}
\makeatother

\def\0{{\bf{0}}}
\def\V{{\bf{V}}}
\def\U{{\bf{U}}}

\def\F{{\bf{F}}}

\def\e{{\bf{e}}}
\def\u{{\bf{u}}}
\def\x{{\bf{x}}}
\def\ul{\underline{\u}}

\def\n{{\bf{n}}}
\def\vv{{\bf{v}}}
\def\vl{\underline{\vv}}

\def\t{{\boldsymbol{\tau}}}

\input{tcilatex}

\excludecomment{confidential}
\restylefloat{table}
\allowdisplaybreaks
\definecolor{db}{rgb}{0.0470,0,0.5294}
\definecolor{dg}{rgb}{0.0,0.392,0.0}
\definecolor{firebrick}{rgb}{0.698,0.133,0.133}
\definecolor{bl}{rgb}{0.0,0.0,0.0}
\definecolor{linen}{rgb}{0.980,0.941,0.902}
\definecolor{ivory}{rgb}{1.0,1.0,0.941}
\definecolor{aliceblue}{rgb}{0.941,0.973,1.0}
\definecolor{beige}{rgb}{0.961,0.961,0.863}
\definecolor{tan}{rgb}{0.824,0.706,0.549}
\definecolor{lightsteelblue}{rgb}{0.690,0.769,0.871}
\definecolor{paleturquoise}{rgb}{0.686,0.933,0.933}
\definecolor{lightblue}{rgb}{0.678,0.847,0.902}
\definecolor{skyblue}{rgb}{0.529,0.808,0.922}
\definecolor{palegoldenrod}{rgb}{0.933,0.910,0.667}
\definecolor{lightgoldenrod}{rgb}{0.933,0.867,0.510}
\definecolor{lightyellow}{rgb}{1.0,1.0,0.878}
\definecolor{yellow}{rgb}{1.0,1.0,0.0}
\definecolor{lightyellow1}{rgb}{1.0,1.0,0.878}
\definecolor{lemonchiffon}{rgb}{1.0,0.980,0.804}
\definecolor{myyellow}{rgb}{1,1,.9}
\definecolor{darkgreen}{rgb}{0.0,0.392,0.0}
\definecolor{darkviolet}{rgb}{0.580,0.0,0.827}
\definecolor{lightsalmon}{rgb}{1.0,0.627,0.478}
\definecolor{orange}{rgb}{1.0,0.647,0.0}

\numberwithin{equation}{section}

\allowdisplaybreaks[1]

\begin{document}

\title
{A variable timestepping algorithm for the unsteady Stokes/Darcy model}

\author{Yi Qin\thanks{%
School of Mathematics and Statistics, Xi'an Jiaotong University, Xi'an, Shaanxi 710049, China.
Email: \href{mailto:qinyi1991@stu.xjtu.edu.cn}{qinyi1991@stu.xjtu.edu.cn}. Subsidized by NSFC (grant No.11971378) and China Scholarship Council grant 201806280136.}
\and Yanren Hou\thanks{%
School of Mathematics and Statistics, Xi'an Jiaotong University, Xi'an, Shaanxi 710049, China.
Email: \href{mailto:yrhou@mail.xjtu.edu.cn}{yrhou@mail.xjtu.edu.cn}. Subsidized by NSFC (grant No.11971378) and China Scholarship Council grant 201806280136.}
\and Wenlong Pei\thanks{%
Department of Mathematics, University of Pittsburgh, Pittsburgh, PA 15260, USA.
Email: \href{mailto:wep17@pitt.edu}{wep17@pitt.edu}. Partially supported by NSF grant DMS 1817542.
}}


\date{\emty}
\maketitle



\begin{abstract}
This report considers a variable step time discretization algorithm proposed by Dahlquist, Liniger and Nevanlinna and applies the algorithm to the unsteady Stokes/Darcy model.
Although long-time forgotten and little explored, the algorithm performs advantages in variable timestep analysis of various fluid flow systems, including the coupled Stokes/Darcy model.
The paper proves that the approximate solutions to the unsteady Stokes/Darcy model are unconditionally stable due to the $G$-stability of the algorithm. Also variable time stepping error analysis
follows from the combination of $G$-stability and consistency of the algorithm. Numerical experiments further verify the theoretical results, demonstrating the accuracy and stability
of the algorithm for time-dependent Stokes/Darcy model.
\end{abstract}



\section{Introduction} \label{sec:introduction}
Stokes/Darcy model, simulating the coupling between surface and subsurface motion of fluid, deserves great interest in geophysics and related areas. Mathematical theory and numerical schemes for both steady and
unsteady Stokes/Darcy model
have been well developed in recent years \cite{badea2010SDmodel,discacciati2002SDmodel,ervin2009coupled,girault2009dg,hou2019solution,layton2002coupling,zuo2015numerical}.
Nevertheless, time discretization for unsteady Stokes/Darcy model is always a big problem where various timestep algorithms give accuracy and efficiency of computation to different levels.
Some simulations use constant timestep, first order, fully implicit scheme for simplicity,
e.g, \cite{ccecsmeliouglu2008SDBE,ccecsmeliouglu2009SDGalerkin,mu2010SDdecoupledmixed,qin2020SDgraddivBE,shan2013SDsubdomain,shan2013SDpartitionedBE}, while many others implement higher
order, constant timestep algorithms to increase accuracy, e.g. \cite{chen2013SDsecondefficient, chen2016SDthirdefficient,layton2012SDCNLF,li2018SDsecondpartition,qin2019SDtimefilterBE}.
Moreover, time stepping adaptivity through variable stepsize schemes is an ideal way of solving the conflict between time accuracy and computational complexity. Due to the limitations of the most existing
methods (e.g. BDF2 is not a $A$-stable under increasing stepsize), variable timestepping analysis for the unsteady Stokes/Darcy model is promising but little studied.

To solve this issue, we refer to a one-parameter family of two-step, one-leg method proposed by Dahlquist, Liniger and Nevanlinna (the DLN method) \cite{DLNStabilityofTwoStepMethods} and apply the method to time-dependent Stokes/Darcy model for variable
timestep analysis. The DLN algorithm maintains the $G$-stability \cite{dahlquist1978positive,dahlquist1976relation,Dahlquist1978,hairer1993solvingII} under any arbitrary sequence of time steps and keeps second order accuracy at same time. To begin with, consider the initial value problem
\begin{align} \label{eq:I.V.P}
x'(t) = f \left(t,x(t)\right), \ \ \ x(0) = x_{0},
\end{align}
where $x: [0,T] \rightarrow \mathbb{R}^{d}$ and $f: \mathbb{R} \times \mathbb{R}^{d} \rightarrow \mathbb{R}^{d}$ are vector-valued functions. Let $\{t_{n}\}_{n=0}^{N}$ be the grids on time interval $[0,T]$
and $k_{n}: = t_{n+1} - t_{n}$ be stepsize. Consequently, we define the stepsize parameter $\epsilon_{n} \in (-1,1)$ to be
\begin{align*}
\epsilon_{n} = \frac{k_{n} - k_{n-1}}{k_{n} + k_{n-1}}.
\end{align*}
Now given the two initial value $x_{0}$ and $x_{1}$, the one parameter DLN algorithm (with parameter $\theta \in [0,1]$) for the problem \eqref{eq:I.V.P} is
\begin{align} \label{eq:DLNscheme}
\sum_{j = 0}^{2} \alpha_{j}x_{n-1+j} = \left(\alpha_{2}k_{n} - \alpha_{0}k_{n-1}\right)f \left( \sum_{j = 0}^{2} \beta_{j,n} t_{n-1+j},\sum_{j = 0}^{2} \beta_{j,n} x_{n-1+j}  \right),
\end{align}
where coefficients $\{\alpha_{j}\}_{j = 0:2}$ (time-independent) and coefficients $\{\beta_{j,n}\}_{j = 0:2}$ (time-dependent) are
\begin{align*}
\begin{bmatrix} \alpha_{2} \vspace{0.2cm} \\ \alpha_{1} \vspace{0.2cm} \\ \alpha_{0} \end{bmatrix}  = \begin{bmatrix} \frac{\theta + 1}{2} \vspace{0.2cm} \\ \theta \vspace{0.2cm} \\ \frac{\theta - 1}{2} \end{bmatrix} ,
\ \ \ \begin{bmatrix} \beta_{2,n} \vspace{0.2cm} \\ \beta_{1,n} \vspace{0.2cm} \\ \beta_{0,n} \end{bmatrix} =
\begin{bmatrix} \frac{1}{4}\left(1+\frac{1-\theta^2}{(1+\epsilon_{n} \theta)^{2}} + \epsilon_{n}^2\frac{\theta(1-\theta^2)}{(1+\epsilon_{n} \theta)^{2}} + \theta \right) \vspace{0.2cm}
\\ \frac{1}{2}\left(1 - \frac{1-\theta^2}{(1+\epsilon_{n} \theta)^{2}} \right) \vspace{0.2cm} \\
\frac{1}{4}\left(1+\frac{1-\theta^2}{(1+\epsilon_{n} \theta)^{2}} - \epsilon_{n}^2\frac{\theta(1-\theta^2)}{(1+\epsilon_{n} \theta)^{2}} - \theta \right) \end{bmatrix}.
\end{align*}
The coefficients of $\{\alpha_{j}\}_{j = 0:2}$, $\{\beta_{j,n}\}_{j = 0:2}$ and the average time step $\alpha_{2}k_{n} - \alpha_{0}k_{n-1}$ are constructed to ensure the $G$-stability and second order accuracy of the method.
Combining these fine properties with existing numerical schemes for spatial discretization (e.g. finite element method \cite{guangzhi2017SDlocalFEMBJ,kanschat2010SDStrongFEM,riviere2005SDdiscontinuousFEM},
two grid decoupled method \cite{hou2016SDoptimaltwogrid,yi2018SDoptimaltwogrid,zuo2018SDparalleltwogrid,zuo2014SDdecouplingtwogrid,zuo2015SDtwogridmixed}, multi-grid decoupled
method \cite{arbogast2009SDdiscretizationmultgrid,zuo2018SDmultigrid}, domain decomposition method \cite{feng2012SDnoniterativeDecomposiion,he2015steadySDdomaindecomposition}, etc.), the paper provides with complete variable timestep analysis for unsteady Stokes/Darcy Model (stability and error analysis).

The reminder of the paper is organized as follows: we review the time dependent Stokes/Darcy model (including necessary notations) in section \ref{sec:SDmodel}. Some preliminaries and two lemmas about
properties of the DLN algorithm \eqref{eq:DLNscheme} are presented in section \ref{sec:preliminaries}.
In section \ref{sec:StabErrorAnalysis}, we apply the variable timestepping DLN algorithm \eqref{eq:DLNscheme} to the unsteady Stokes/Darcy model and provide with detailed proofs of unconditional stability and second order convergence of approximate solutions, which are rarely done in other papers. Two numerical tests are given in section \ref{sec:numtests}. The variable timestepping test is aimed to verify the stability of
the approximate solutions and followed by a constant timestepping example to confirm the second order convergence by the DLN algorithm.

\subsection{Related Works} \label{subsec:relatedworks}
The variable timestepping analysis on computational fluid flow is little understood due to limitations of the most existing time discretization schemes. The DLN timestepping algorithm, which is second order, unconditionally $G$-stable under variable time steps, has been applied to the Navier-Stokes equations for variable timestep stability and error analysis \cite{2020arXiv200108640L}. However the first choice for variable timestep analysis of
fluid flow is the first order fully implicit method (backward Euler method) for its simplicity and unconditional stability. Recently backward Euler method has been used in artificial compression algorithm with adaptivity for the Navier-Stokes equations \cite{Layton:2019:DoublyAdaptiveAC}. Furthermore it is possible that adding time filters on the backward Euler method increases the order of convergence while keeping the conditional stability for fluid flow \cite{Guzel2018TimeFI}.

\section{The Time-dependent Stokes/Darcy Model} \label{sec:SDmodel}
In this section, we consider the unsteady Stokes/Darcy model in region $\Omega=\Omega_f \cup \Omega_p$, where $\Omega_f$ is the incompressible fluid region and $\Omega_p$ is the porous media region.
The two regions are separated by the interface denoted by $\Gamma=\overline{\Omega}_f\cap\overline{\Omega}_p$ and $\n_f$ and $\n_p$ are the unit outward normal vectors on $\partial\Omega_f$ and $\partial\Omega_p$.
The schematic representation is displayed in Figure \ref{fig1}.
\begin{figure}[ptbh]
\centering
\par
{
		\begin{minipage}[t]{0.45\linewidth}
			\centering
			\includegraphics[width=2.3in,height=1.8in]{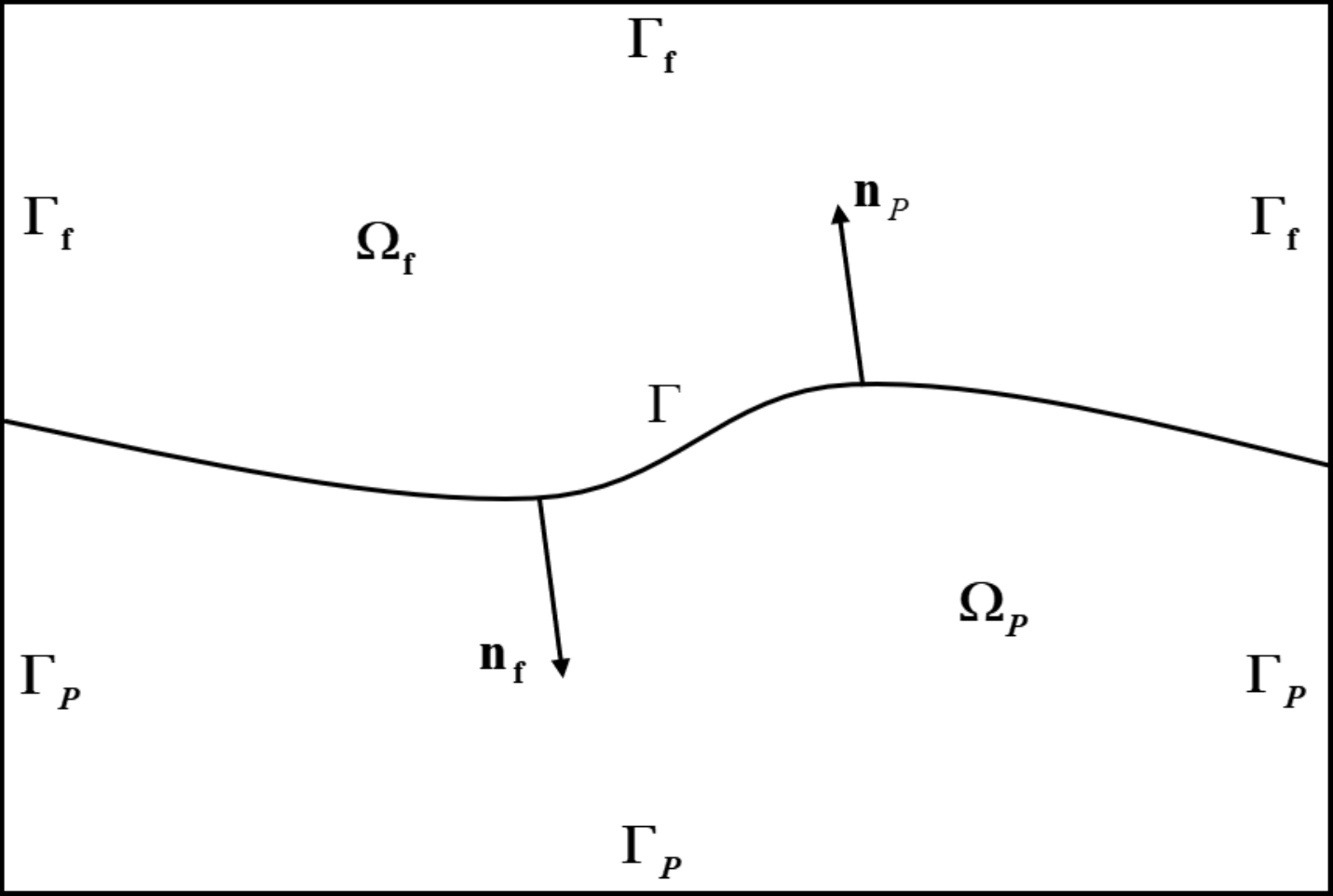}\\
			\vspace{0.02cm}
					\end{minipage}	} \centering
\vspace{-0.2cm}
\caption{A global domain $\Omega$ consisting of a fluid flow region $\Omega_f$ and a porous media flow region $\Omega_p$ is separated by an interface $\Gamma$.}
\label{fig1}
\end{figure}

For the finite time interval $[0, T]$, the fluid motion in $\Omega_f$  is governed by the time-dependent Stokes equations, i.e. the fluid velocity $\u_{f}(\x,t)$ and the pressure $p(\x,t)$ satisfy
\begin{align}\label{eq:EqStokes}
\frac{\partial \u_f}{\partial t} -\nabla \cdot \mathds{T}\left( \u_{f}, p \right) &= \mathbf{F}_{1}(\x,t) \ \ \ \ \ \ \text{in\ } \Omega_f\times(0,T), \notag \\
\nabla\cdot \u_f &= 0  \ \ \ \ \ \ \ \ \ \ \ \ \ \ \ \ \text{in\ } \Omega_f\times(0,T), \\
\u_{f}(\x,0) &= \u_{f}^0(\x)   \ \ \ \ \ \ \ \ \ \text{in\ } \Omega_f , \notag
\end{align}
where the stress tensor $\mathds{T}$ and the deformation rate tensor $\mathds{D}$ are defined as
\begin{align*}
\mathds{T}(\u_{f},p)=-p\mathds{I}+2\nu\mathds{D}(\u_{f}), \ \ \ \mathds{D}(\u_{f})=\frac12(\nabla \u_{f}+\nabla^{\text{tr}} \u_{f}) \footnotemark[1].
\end{align*}
\footnotetext[1]{$\nabla^{\text{tr}} \u_{f}$ means the transpose of the tensor $\nabla \u_{f}$.}
\noindent $\nu>0$ is the kinetic viscosity, $\mathds{I}$ represents the identity matrix and $\mathbf{F}_{1}$ is the external force. The velocity $\u_{p}(\x,t)$ and hydraulic head $\phi(\x,t)$ in porous media region are governed by the Darcy's law and the saturated flow model
\begin{align}
\u_{p} =-\mathbf{K}\nabla\phi & \ \ \ \ \ \ \ \ \text{in\ } \Omega_p\times(0,T),  \label{eq:DarcyLaw}\\
S_0 \frac{\partial\phi}{\partial t} +\nabla\cdot \u_{p} = \rm{F}_{2}(\x,t) & \ \ \ \ \ \ \ \ \text{in\ }\Omega_p\times(0, T),  \label{eq:SaturatedModel}\\
\phi(\x,0) =\phi^0(\x) & \ \ \ \ \ \ \ \ \text{in\ }\Omega_p,  \notag
\end{align}
where positive symmetric tensor $\mathbf{K}$ denotes the hydraulic conductivity in $\Omega_p$ and is allowed to vary in space. $S_0$ is the specific mass storativity coefficient and $\rm{F}_{2}$ is a source term.
Combining \eqref{eq:DarcyLaw} and \eqref{eq:SaturatedModel}, we obtain the Darcy equation which describes the hydraulic head:
\begin{align} \label{eq:DarcyEq}
S_0 \frac{\partial\phi}{\partial t} - \nabla\cdot(\mathbf{K}\nabla\phi)=\rm{F}_{2}(x,t), \ \ \ \ \ \ \ \ \text{in\ }\Omega_p\times(0,T),
\end{align}
Now we introduce the boundary conditions:
\begin{gather}
\u_{f} = 0 \ \ \ \text{on\ } \left(\partial\Omega_f\setminus\Gamma \right) \times(0,T) , \notag \\
\ \phi=0 \ \ \ \text{on\ } \left( \partial\Omega_p\setminus \Gamma \right) \times(0,T),  \label{eq:BoundaryCondi}
\end{gather}
and the necessary interface conditions for the coupled Stokes/Darcy model:
\begin{align}
\u_{f} \cdot \n_f - \mathbf{K}\nabla\phi\cdot \n_p &= 0,   \ \ \ \ \ \ \ \ \ \ \ \ \ \ \ \ \ \ \ \ \ \ \ \ \ \ \ \ \ \text{on\ } \Gamma\times(0,T), \notag\\
- \n_f \cdot \Big( \mathds{T}(\u_{f},p)\cdot \n_f \Big) &= g\phi,  \ \ \ \ \ \ \ \ \ \ \ \ \ \ \ \ \ \ \ \ \ \ \ \ \ \ \ \text{on\ } \Gamma\times(0,T),  \label{eq:BJScondi}\\
- \t_i \cdot \Big( \mathds{T}(\u_{f},p)\cdot \n_f \Big) &= \frac{\mu_{BJS}\nu\sqrt{d}}{\sqrt{\text{trace}(\Pi)}}  \t_i \cdot \u_{f},
\ \ \ i = 1, \cdots d-1 \ \ \ \ \ \ \text{on\ } \Gamma\times(0,T),  \notag
\end{align}
where $g$ is the gravitational constant and $\{\t_i\}_{i=1}^{d-1}$ are the orthonormal system of tangential vectors along $\Gamma$. $\mu_{BJS}$ is an experimentally determined parameter. $\Pi$ represents the permeability and satisfies $\mathbf{K}=\frac{\Pi g}{\nu}$.

For weak formulation of the unsteady Stokes/Darcy model, we define some function spaces:
\begin{align*}
H_{f} &= \{ \vv \in \left(H^1(\Omega_f) \right)^{d}: \vv|_{\Omega_f \setminus \Gamma}=0\}, \\
H_{p} &= \{ \psi \in H^1(\Omega_p): \psi |_{\Omega_p \setminus \Gamma}=0\} , \\
\U &= H_f \times H_p , \\
Q_f &=L^2(\Omega_f).
\end{align*}
We associate the space $\U$ with the following two norms: for all $\vl=(\vv, \psi)\in \U$
\begin{align*}
\Vert \vl \Vert_0 &=\sqrt{(\vv,\vv)_{\Omega_f}+gS_0(\psi,\psi)_{\Omega_p}},\\
\Vert \vl \Vert_{\U} &=\sqrt{\nu(\nabla \vv,\nabla \vv)_{\Omega_f} + g (\mathbf{K} \nabla \psi,\nabla \psi)_{\Omega_p}},
\end{align*}
where $\left( \cdot , \cdot  \right)_{\Omega}$ denotes the $L^{2}$-inner product on function space $L^{2}(\Omega)$. 
By positive definiteness of tensor $\mathbf{K}$ and Poincar\'e inequality, there exists constant $C_{0,\U} > 0$ such that
\begin{align} \label{eq:EquiNorm}
\Vert \vl \Vert_0 \leq C_{0,\U} \Vert \vl \Vert_{\U}.
\end{align}
For convenience, we denote $\Vert \cdot \Vert$ and $\Vert \cdot \Vert_{k}$ are norms of $L^{2}$ space and Sobolev space $H^{k}$ respectively.

Now we combine \eqref{eq:EqStokes}, \eqref{eq:DarcyEq}, \eqref{eq:BoundaryCondi} and \eqref{eq:BJScondi} to derive the weak form of time dependent Stokes/Darcy model: given
$\mathbf{F} = (\mathbf{F}_{1},  \rm{F}_{2}) \in L^2 \left(0,T; \left(L^{2}(\Omega_f)\right)^{d} \right) \times L^2 \left(0,T;  L^2(\Omega_p) \right)$, find
$\ul(t) = (\u_{f}(t), \phi (t)) \in \U $ and $p(t) \in Q_{f}$ such that for all $ \vl = (\vv, \psi) \in \U$, $q \in Q_{f}$ and any time $t \in (0,T]$
\begin{gather}
\Big<\frac{\partial \ul}{\partial t},\vl \Big>_{0}+a(\ul,\vl)+b(\vl,p)= \left<\F,\vl \right>_{\U'}, \notag \\
b(\ul,q)=0,   \label{eq:WeakSD}\\
\ul(\x,0)=\ul^0,   \notag
\end{gather}
where
\begin{align*}
\Big<\frac{\partial\ul}{\partial t},\vl \Big>_{0} &= \left(\frac{\partial \u_{f}}{\partial t},\vv \right)_{\Omega_f}+ gS_0 \left(\frac{\partial\phi}{\partial t},\psi \right)_{\Omega_p},\\
a(\ul,\vl) &= a_{\Omega}(\ul,\vl)+a_\Gamma(\ul,\vl),\\
a_{\Omega}(\ul,\vl) &= a_{\Omega_f}(\u,\vv)+a_{\Omega_p}(\phi,\psi),\\
a_{\Omega_f}(\u,\vv) &= \nu \left(\mathds{D}(\u),\mathds{D}(\vv) \right)_{\Omega_f}+ \left( \frac{\mu_{BJS} \nu\sqrt{d}}{\sqrt{\text{trace}(\Pi)}} P_\t(\u), \vv \right)_\Gamma,\\
a_{\Omega_p}(\phi,\psi) &= g(\mathbf{K}\nabla \phi,\nabla \psi)_{\Omega_p},\\
a_\Gamma(\ul,\vl) &= g \left( \phi,\vv\cdot \n_s \right)_\Gamma - g \left(\psi,\u_{f} \cdot \n_s \right)_\Gamma,\\
b(\vl,p) &= -(p,\nabla\cdot \vv)_{\Omega_f},\\
\left< \mathbf{F},\vl \right>_{\U'} &= (\mathbf{F}_{1},\vv)_{\Omega_f}+g(\rm{F}_{2},\psi)_{\Omega_p}, \\
\ul^0 &= (\u_{f}^0(\x),\phi^0(\x)).
\end{align*}
$\U'$ is the dual space of $\U$ with the norm
\begin{align*}
\Vert \F \Vert_{\U'} = \sup_{\vv \in \U \backslash \{0\}} \frac{\left< \mathbf{F},\vl \right>_{\U'}}{\Vert \vl \Vert_{\U}},
\end{align*}
and $P_\t(\cdot)$ is the projection onto the local tangential plane, i.e. $P_\t(\vv)=\vv-(\vv\cdot\n_s)\n_s$.
The bilinear form $a(\cdot, \cdot)$ is continuous and coercive: for all $\ul,\vl\in \U$,
\begin{align}
a(\ul,\vl) & \leq C_{1} \|\ul\|_U\|\vl\|_U,    \notag\\
a(\ul,\ul) & \geq C_{2} \|\ul\|_U^2.  \label{eq:aEstimator}
\end{align}
Here the above constants $C_{1},C_{2}> 0$ are independent of functions.
\section{Preliminaries} \label{sec:preliminaries}
For spatial discretization, we construct regular triangulations of $\Omega_f$ and $\Omega_p$ with diameter $h >0$ and choose any finite element spaces $H_{fh} \subset H_f, Q_{fh} \subset Q_f , H_{ph} \subset H_p $
such that the pair $\left( H_{fh}, Q_{fh} \right)$ satisfies the discrete $LBB^{h}$ condition. Typical examples of such pair include Taylor-Hood (P2-P1) and MINI (P1b-P1). Then we define $\U_h= H_{fh} \times H_{ph}$ to be
finite element space of $\U$. The discretely divergence free subspace of $H_{fh}$ is defined to be
\begin{align*}
V_{fh} & := \left\{ \vv_{h} \in H_{fh} : \left( \nabla \cdot \vv_{h}, q_{h} \right) = 0, \ \forall q_{h} \in Q_{fh} \right\},
\end{align*}
and the divergence free space of $\U_{h}$ to be $\V_{h} = V_{fh} \times H_{ph}$.

We define the linear projection operator (see \cite{mu2010SDdecoupledmixed}) $P_h = \left(P_h^{\ul}, P_h^{p} \right)$ from $\U\times Q_f$ onto $\U_h\times Q_{fh}$ :
given $t\in(0,T]$ and $\left(\ul(t),p(t)\right) \in \left(\U, Q_{f} \right)$, $\left( P_h^{\ul}\u(t), P_h^{p}p(t) \right)$ satisfies
\begin{gather}
a(P_h^{\ul}\ul(t),\vl_h)+b(\vl_h,P_h^{p}p(t))=a(\ul(t),\vl_h)+b(\vl_h,p(t)), \notag \\
b(P_h^{\ul}\ul(t),q_{h})=0.    \label{projection}
\end{gather}
for all $\vl_h\in\U_h$, $q_{h}\in Q_{fh}$. and the linear projection $P_h$ defined above satisfies
\begin{align}
\Vert P_h^{\ul}\ul(t)-\ul(t) \Vert_0 &\leq C_{3}h^2 \Vert \ul(t) \Vert_{2}, \notag\\
\Vert P_h^{\ul}\ul(t)-\ul(t) \Vert_U &\leq C_{4}h \Vert\ul(t) \Vert_{2},   \notag  \\
\Vert P_h^{p}p(t)-p(t) \Vert &\leq C_{5}h \Vert p(t) \Vert_{1}.  \label{eq:ProjProperty}
\end{align}
if the pair $\left( \ul(t),p(t)\right)$ is smooth enough.

For the rest of the paper, $P = \{t_{n}\}_{n=0}^{N}$ is the partition on time interval $[0,T]$ with $t_{0} = 0, t_{N} = T$ and $k_{n} = t_{n+1} - t_{n}$ is the time stepsize.
Let $\ul_{h}^{n}, p_{h}^{n}$ denote the approximate solutions of $\ul(t_{n}), p(t_{n})$ by the DLN method \eqref{eq:DLNscheme} and for convenience, we denote
\begin{align*}
\ul_{h,\beta}^{n} &= \beta_{2,n} \ul_{h}^{n+1}+\beta_{1,n} \ul_h^{n}+\beta_{0,n} \ul_h^{n-1}, \\
\F_{\beta}^{n} &= \beta_{2,n} \F(t_{n+1})+\beta_{1,n} \F(t_{n})+\beta_{0,n} \F(t_{n-1}).
\end{align*}
Then we have the discrete weak formulation for the unsteady Stokes/Darcy model by variable timestepping DLN algorithm:
given $\ul_h^{n}$, $\ul^{n-1}_h$ and $p_h^{n}$, $p^{n-1}_h$, find $\ul_{h}^{n+1}, p_{h}^{n+1}$ such that for all $\vl_h \in \U_{h}$ and $q_h\in Q_{fh}$,
\begin{gather}
\left<\frac{\alpha_2\ul_h^{n+1}+\alpha_1\ul_h^{n}+\alpha_0\ul_h^{n-1}}{\alpha_2k_n-\alpha_0k_{n-1}},\vl_h \right>_{0} + a(\ul_{h,\beta}^{n},\vl_h)+b(\vl_h,p^{n}_{h,\beta})=\left<\F^{n}_{\beta},\vl_h \right>_{\U'}, \notag \\
b(\ul_h^{n+1},q_h)=0.   \label{eq:DLNSDWeakForm1}
\end{gather}
Under $LBB^{h}$ condition, \eqref{eq:DLNSDWeakForm1} has equivalent form: for all $\vl_h \in \V_{h}$
\begin{align}\label{eq:DLNSDWeakForm2}
\left< \frac{\alpha_2 \ul_h^{n+1}+\alpha_1 \ul_h^{n}+\alpha_0 \ul_h^{n-1}}{\alpha_2k_n-\alpha_0k_{n-1}},\vl_h \right>_{0} + a(\ul_{h,\beta}^{n},\vl_h ) = \left< \F^{n}_{\beta},\vl_h \right>_{\U'}.
\end{align}
Before proceeding to next section, we propose two lemmas about the DLN method needed for stability and error analysis.
\begin{lemma} \label{lemma:GstabDLN}
The $DLN$ scheme \eqref{eq:DLNscheme} under variable timestep is G-stable, i.e. for any $n = 1,2,..., N-1$, there exist real numbers $\lambda_{j,n} \ (j = 0,1,2) $  such that
\begin{align*}
\left(\sum_{j=0}^{2} {\alpha_{j}} x_{n-1+j}, \sum_{j=0}^{2} {\beta_{j,n}} x_{n-1+j} \right) = \begin{Vmatrix} {x_{n+1}} \\ {x_{n}} \end{Vmatrix}^{2}_{G(\theta)}-\begin{Vmatrix} {x_{n}} \\ {x_{n-1}} \end{Vmatrix}^{2}_{G(\theta)}+ \left\Vert \sum_{j=0}^{2} {\lambda_{j,n}} x_{n-1+j} \right\Vert ^{2} .
\end{align*}
Here the $G(\theta)$-norm $\Vert \cdot \Vert_{G \left( \theta \right)}$ (timestep independent norm) is
\begin{align} \label{eq:defGnorm}
\begin{Vmatrix} y \\ z \end{Vmatrix}^{2}_{G(\theta)} := {\frac{1}{4}}(1 + {\theta}) \left\Vert y \right\Vert^{2} + {\frac{1}{4}}(1 - {\theta}) \left\Vert z \right\Vert^{2},
\end{align}
for any $y,z \in \mathbb{R}^{d}$ and the coefficients $\lambda_{j,n}$ in numerical dissipation are
\begin{align*}
\lambda_{1,n} = - \frac{\sqrt{\theta \left( 1 - {\theta}^{2} \right)}}{\sqrt{2} \left( 1 + \epsilon_n \theta \right)} , \ \ \ \lambda_{2,n} = - \frac{1 - \epsilon_n}{2} \lambda_{1,n} , \ \ \ \lambda_{0,n} = - \frac{1 + \epsilon_n}{2} \lambda_{1,n}.
\end{align*}
\end{lemma}
\begin{proof}
See \cite{2020arXiv200108640L}.
\end{proof}
\ \\

\begin{lemma} \label{lemma:2ndDLN}
Let $y : \Omega \times [0,T] \to \mathbb{R}^{d} $ be smooth enough, then
\begin{align*}
\left\Vert \sum_{j=0}^{2} \beta_{j,n}y(t_{n-1+j}) - y \left( t_{n,\beta} \right)  \right\Vert^{2} \leq C \left( k_{n} + k_{n-1} \right)^{3}  \int_{t_{n-1}}^{t_{n+1}} \left\Vert y_{tt} \right\Vert^{2} dt ,
\end{align*}
and for $\theta \in [0,1)$
\begin{align*}
\left\Vert \frac{{\alpha_{2}}{y(t_{n+1})} + {\alpha_{1}}{y(t_{n})} + {\alpha_{0}}{y(t_{n-1})}}{\alpha_2k_n-\alpha_0k_{n-1}} - y_{t} \left( t_{n,\beta} \right) \right\Vert^{2} \leq C \left( \theta \right) \left( k_{n} + k_{n-1} \right)^{3} \int_{t_{n-1}}^{t_{n+1}} \Vert y_{ttt} \Vert^{2} dt ,
\end{align*}
where $t_{n,\beta}=\beta_{2,n}t_{n+1}+\beta_{1,n}t_n+\beta_{0,n} t_{n-1}$.
\end{lemma}
\begin{proof}
Apply Taylor theorem with integral reminder to $y(t_{n+1})$, $y(t_{n-1})$ and $y \left( t_{n,\beta} \right)$ and expand these functions at point $t_{n}$.
\end{proof}

\section{Variable timestepping Analysis for the Unsteady Stokes/Darcy Model} \label{sec:StabErrorAnalysis}
Now we apply $G$-stability of the DLN method (Lemma \ref{lemma:GstabDLN}) and have the following theorem about stability of approximate solutions by variable timestepping DLN algorithm.
\begin{theorem}  (Unconditional Stability)
\label{theorem:SDStabDLN}
For any $2 \leq M \leq N$, the approximate solutions of the unsteady Stokes/Darcy model by the algorithm \eqref{eq:DLNSDWeakForm2} satisfy
\begin{gather}
\frac{1}{4}(1+\theta)\left\Vert \ul_{h}^{M} \right\Vert^2_0 + \frac{1}{4}(1-\theta)\left\Vert \ul_{h}^{M-1} \right\Vert^2_0+\sum_{n=1}^{M-1} \left\Vert \sum_{j=0}^{2} \lambda_{j,n}\ul_{h}^{n-1+j} \right\Vert^2_0
+C(\theta) \sum_{n=1}^{M-1}(k_n+k_{n-1})\left\Vert \ul_{h,\beta}^{n} \right\Vert_U^2 \notag \\
\leq  \frac{1}{4}(1+\theta) \left\Vert \ul_{h}^{1} \right\Vert^2_0+\frac{1}{4}(1-\theta)\left\Vert \ul_{h}^{0} \right\Vert^2_0 + \widetilde{C}(\theta)\sum_{n=1}^{N-1}(k_n+k_{n-1}) \left\Vert \F^{n}_{\beta} \right\Vert_{U'}^2. \label{eq:SDsoluStab}
\end{gather}
Here, the constants $C(\theta), \widetilde{C}(\theta) \geq 0$ are independent of the diameter $h$ and time stepsize $k_n$.
\end{theorem}
\begin{proof}
Let $\vl_h=\ul_{h,\beta}^{n}$ in \eqref{eq:DLNSDWeakForm2} and multiply both sides of the equation by $\alpha_2 k_n-\alpha_0 k_{n-1}$,
\begin{align} \label{eq:SDStabstep1}
\left< \sum_{j=0}^{2}\alpha_j \ul_h^{n-1+j},\sum_{j=0}^{2}\beta_{j,n} \ul_{h}^{n-1+j} \right>_{0}+\left(\alpha_2 k_n-\alpha_0 k_{n-1} \right)a(\ul_{h,\beta}^{n},\ul_{h,\beta}^{n})
=\left(\alpha_2 k_n-\alpha_0 k_{n-1}\right) \left<\F_{\beta}^{n},\ul_{h,\beta}^{n} \right>_{\U'}.
\end{align}
Using Lemma \ref{lemma:GstabDLN} for \eqref{eq:SDStabstep1} and replacing $L^{2}$ space by $\U$ and $L^{2}$-norm by $\Vert \cdot \Vert_{0}$ norm, we obtain
\begin{align} \label{eq:SDGstabterm}
\left< \sum_{j=0}^{2}\alpha_j \ul_h^{n-1+j},\sum_{j=0}^{2}\beta_{j,n} \ul_{h}^{n-1+j} \right>_{0} = \begin{Vmatrix} {\ul_{h}^{n+1}} \\ {\ul_{h}^{n}} \end{Vmatrix}^{2}_{G(\theta)}-\begin{Vmatrix} {\ul_{h}^{n}} \\ {\ul_{h}^{n-1}} \end{Vmatrix}^{2}_{G(\theta)}+ \left\Vert \sum_{j=0}^{2} {\lambda_{j,n}} \ul_{h}^{n-1+j} \right\Vert_{0}^{2} ,
\end{align}
where $\Vert \cdot \Vert_{0}$ is the norm induced by inner product $\left< \cdot , \cdot \right>_{0}$ and the corresponding $G(\theta)$-norm becomes
\begin{align} \label{eq:SDGstabdef}
\begin{Vmatrix} \ul_{h}^{n+1} \\ \ul_{h}^{n} \end{Vmatrix}^{2}_{G(\theta)} = {\frac{1}{4}}(1 + {\theta}) \left\Vert \ul_{h}^{n+1} \right\Vert_{0}^{2} + {\frac{1}{4}}(1 - {\theta}) \left\Vert \ul_{h}^{n} \right\Vert_{0}^{2}.
\end{align}
Then we apply \eqref{eq:aEstimator}, \eqref{eq:SDGstabterm} and Cauchy Schwarz inequality to \eqref{eq:SDStabstep1}:
\begin{gather}
\begin{Vmatrix} \ul_{h}^{n+1} \\ \ul_{h}^{n} \end{Vmatrix}^{2}_{G(\theta)} - \begin{Vmatrix} \ul_{h}^{n} \\ \ul_{h}^{n-1} \end{Vmatrix}^{2}_{G(\theta)}
+ \left\Vert \sum_{j=0}^{2} \lambda_{j,n}\ul_{h}^{n-1+j} \right\Vert_{0}^2 + C_{2}(\alpha_2 k_n-\alpha_0 k_{n-1})\|\ul_{h,\beta}^{n}\|_U^2  \notag \\
\leq \frac{C_{2}}{2}(\alpha_2 k_n-\alpha_0 k_{n-1})\|\ul_{h,\beta}^{n}\|_U^2+\frac{1}{2C_{2}}(\alpha_2 k_n-\alpha_0 k_{n-1}) \left\Vert\F_{\beta}^{n} \right\Vert_{U'}^2.  \label{eq:SDStabstep2}
\end{gather}
Note that
\begin{align}  \label{eq:EstAvestep}
\frac{1-\theta}{2} \left( k_{n}+k_{n-1} \right) \leq \alpha_2 k_n-\alpha_0 k_{n-1} \leq \frac{1+\theta}{2} \left( k_{n} + k_{n-1} \right),
\end{align}
\eqref{eq:SDStabstep2} becomes
\begin{align}
\begin{Vmatrix} \ul_{h}^{n+1} \\ \ul_{h}^{n} \end{Vmatrix}^{2}_{G(\theta)} - \begin{Vmatrix} \ul_{h}^{n} \\ \ul_{h}^{n-1} \end{Vmatrix}^{2}_{G(\theta)}
+ \left\Vert \sum_{j=0}^{2} \lambda_{j,n}\ul_{h}^{n-1+j} \right\Vert_{0}^2 + \frac{C_{2}(1-\theta)}{4}(k_n + k_{n-1})\left\Vert \ul_{h,\beta}^{n} \right\Vert_U^2
\leq \frac{1+\theta}{4C_{2}}(k_n+k_{n-1}) \left\Vert \F_{\beta}^{n} \right\Vert_{U'}^2. \label{eq:SDStabstep3}
\end{align}
Summing over \eqref{eq:SDStabstep3} from $n = 1, \cdots, M-1$ and using \eqref{eq:SDGstabdef}, we obtain \eqref{eq:SDsoluStab}.
\end{proof}
\ \\

Next we apply $G$-stability ( Lemma \ref{lemma:GstabDLN}) and consistency (Lemma \ref{lemma:2ndDLN}) properties of the DLN algorithm to show the second order convergence of approximate solutions to unsteady Stokes/Darcy model.
We denote $\ul^{n}=(\u_{f}^{n},\phi^{n})$ and $p^{n}$ be the exact solutions of the coupled Stokes/Darcy model \eqref{eq:WeakSD} at time $t_{n}$ and define the error functions to be
\begin{align}
\e^{n}&=\ul_{h}^{n}-\ul^{n}=(\ul_{h}^{n}-P_h^{\ul}\ul^n)-(\ul^n-P_h^{\ul}\ul^n)=\eta^n-\xi^n. \notag \\
e_{p}^{n}&=p_{h}^{n}-p^{n}=(p_{h}^{n}-P_h^{p}p^{n})-(p^{n}-P_h^{p}p^{n})=\eta_{p}^n-\xi_{p}^n,   \label{eq:defErrfunc}
\end{align}
and $\eta^0=\eta^1=0$.
For variable timestepping analysis, we need to define some continuous and discrete norms. Given $\vl \in \U$, $q \in Q_{f}$, $\mathbf{G} \in \U'$ and $1 \leq m,s < \infty$, we define continuous norms
\begin{gather*}
\left\Vert \vl \right\Vert_{m,0} := \left( \int_{0}^{T} \left\Vert \vl(t) \right\Vert_{0}^{m} dt \right)^{1/m}, \ \ \
\left\Vert \vl \right\Vert_{m,s} := \left( \int_{0}^{T} \left\Vert \vl(t) \right\Vert_{s}^{m} dt \right)^{1/m}, \ \ \
\left\Vert \vl \right\Vert_{m,\U} := \left( \int_{0}^{T} \left\Vert \vl(t) \right\Vert_{\U}^{m} dt \right)^{1/m}.  \\
\left\Vert q \right\Vert_{m,L^{2}} := \left( \int_{0}^{T} \left\Vert q(t) \right\Vert^{m} dt \right)^{1/m}, \ \ \
\left\Vert \mathbf{G} \right\Vert_{m,\U'} := \left( \int_{0}^{T} \left\Vert G(t) \right\Vert_{\U'}^{m} dt \right)^{1/m}.
\end{gather*}
and new discrete norms
\begin{gather*}
\left\Vert \left| \vl \right| \right\Vert_{m,0} := \left( \sum_{n=0}^{N-1} k_{n} \left\Vert \vl^{n+1} \right\Vert_{0}^{m} \right)^{1/m}, \ \ \
\left\Vert \left| \vl \right| \right\Vert_{m,s} := \left( \sum_{n=0}^{N-1} k_{n} \left\Vert \vl^{n+1} \right\Vert_{s}^{m} \right)^{1/m}, \\
\left\Vert \left| \vl_{\beta} \right| \right\Vert_{m,s} := \left( \sum_{n=0}^{N-1} \left(k_{n-1} + k_{n} \right) \left\Vert \vl \left(t_{n,\beta} \right) \right\Vert_{s}^{m} \right)^{1/m}.
\end{gather*}
Now we have the main theorem for error analysis.
\begin{theorem} (Second order convergence)
\label{theorem:DLNSDErrorEsti}
The approximate solutions $\{\ul_{h}^{n}\}_{n=0}^{N}$ by  the variable timestepping DLN scheme \eqref {eq:DLNSDWeakForm2} with parameter $\theta \in [0,1)$ satisfy
\begin{align}
\left\Vert \left| \ul_{h}-\ul \right| \right\Vert_{2,0} \leq & C(\theta)
\bigg\{ \max_{1\leq n\leq N-1}\left\{ ( k_{n} + k_{n-1} )^{2} \right\} \left( \left\Vert p_{tt} \right\Vert_{2,L^{2}} + \left\Vert \ul_{ttt} \right\Vert_{2,0} + \left\Vert \ul_{tt} \right\Vert_{2,\U} + \left\Vert \F_{tt} \right\Vert_{2,\U'} \right) \notag \\
&\ \ \ \ \ \ \ \ \ \ \ \ \ \ + h^{2} \left\Vert \ul_{t} \right\Vert_{2,2} + h^{2} \left\Vert \left| \ul \right| \right\Vert_{2,2} \bigg\}, \label{eq:DLNSDErrEstL2}
\end{align}
and
\begin{gather}
\left(\sum_{n=1}^{N-1} \left(\alpha_{2}k_{n} - \alpha_{0}k_{n-1}\right) \left\Vert \ul(t_{n,\beta}) - \ul_{h,\beta}^{n} \right\Vert_{\U}^{2} \right)^{1/2} \notag \\
\leq  C(\theta) \max_{1\leq n\leq N-1}\left\{ ( k_{n} + k_{n-1} )^{2} \right\} \left( \left\Vert p_{tt} \right\Vert_{2,L^{2}} + \left\Vert \ul_{ttt} \right\Vert_{2,0} + \left\Vert \ul_{tt} \right\Vert_{2,\U} + \left\Vert \F_{tt} \right\Vert_{2,\U'}  \right) + C(\theta) h^{2} \left\Vert \ul_{t} \right\Vert_{2,2} \notag \\
+ C(\theta) h \max_{1\leq n\leq N-1}\left\{ ( k_{n} + k_{n-1} )^{2} \right\} \left\Vert \ul_{tt} \right\Vert_{2,2} + C(\theta) h \left\Vert \left| \ul_{\beta} \right| \right\Vert_{2,2}.  \label{eq:DLNSDErrEstH1Final}
\end{gather}
\end{theorem}
\begin{proof}
By \eqref{eq:WeakSD}, the true solutions of unsteady Stokes/Darcy model at time $t_{n,\beta}$. By \eqref{eq:WeakSD} satisfy
\begin{align} \label{eq:DLNSDError1}
\left<\frac{\partial\ul}{\partial t}(t_{n,\beta}),\vl_h \right>_{0}+a(\ul(t_{n,\beta}),\vl_h)+b(\vl_h,p(t_{n,\beta}))=\left< \F(t_{n,\beta}),\vl_h \right>_{\U'}, \ \ \ \text{for all }\ \vl_h\in\V_{h}.
\end{align}
Equivalently, \eqref{eq:DLNSDError1} can be rewritten as
\begin{align}  \label{eq:DLNSDError2}
\left<\frac{\alpha_2\ul^{n+1}+\alpha_1\ul^{n}+\alpha_0\ul^{n-1}}{\alpha_2k_n-\alpha_0k_{n-1}},\vl_h \right>_{0}+a(\ul_{\beta}^{n},\vl_h)+b(\vl_h,p^{n}_{\beta})
=\left< \F_{\beta}^{n},\vl_h \right>_{\U'}+\tau(\ul(t_{n,\beta}),p(t_{n,\beta}),\vl_h),
\end{align}
where
\begin{gather*}
\ul_{\beta}^{n} = \beta_{2,n} \ul^{n+1}+\beta_{1,n} \ul^{n}+\beta_{0,n} \ul^{n-1}, \ \ \  p_{\beta}^{n} = \beta_{2,n} p^{n+1}+\beta_{1,n} p^{n}+\beta_{0,n} p^{n-1}, \\
\tau(\ul(t_{n,\beta}),p(t_{n,\beta}),\vl_h) = \left<\frac{\alpha_2\ul^{n+1}+\alpha_1\ul^{n}+\alpha_0\ul^{n-1}}{\alpha_2k_n-\alpha_0k_{n-1}}-\frac{\partial\ul}{\partial t}(t_{n,\beta}),\vl_h \right>_{0}
+a(\ul_{\beta}^{n}-\ul(t_{n,\beta}),\vl_h) \\
+b(\vl_h, p^{n}_{\beta}-p(t_{n,\beta}))- \left<\F^{n}_{\beta}-\F(t_{n,\beta}),\vl_h \right>_{\U'}.
\end{gather*}
Note that $\V_{h} \subset \U_{h}$, the system \eqref{eq:DLNSDWeakForm1} holds for all $\vl_h\in\V_{h}$. Thus we subtract \eqref{eq:DLNSDError2} from first equation of \eqref{eq:DLNSDWeakForm1} and use the definition of error
function in \eqref{eq:defErrfunc} to obtain: for all $\vl_h\in\V_{h}$,
\begin{gather}
\left< \frac{\alpha_2\eta^{n+1}+\alpha_1\eta^{n}+\alpha_0\eta^{n-1}}{\alpha_2k_n-\alpha_0k_{n-1}},\vl_h \right>_{0} + a(\eta_{\beta}^{n},\vl_h)+b(\vl_h,\eta_{p,\beta}^{n}) \notag \\
=\left< \frac{\alpha_2\xi^{n+1}+\alpha_1\xi^{n}+\alpha_0\xi^{n-1}}{\alpha_2k_n-\alpha_0k_{n-1}},\vl_h \right>_{0} + a(\xi_{\beta}^{n},\vl_h)+b(\vl_h,\xi^{n}_{p,\beta})
-\tau \left(\ul(t_{n,\beta}),p(t_{n,\beta}),\vl_h \right), \label{eq:DLNSDError3}
\end{gather}
where
\begin{align*}
\eta_{\beta}^{n} &= \beta_{2,n}\eta^{n+1}+\beta_{1,n}\eta^{n}+\beta_{0,n}\eta^{n-1}, \ \ \ \ \ \ \xi_{\beta}^{n} = \beta_{2,n}\xi^{n+1}+\beta_{1,n}\xi^{n}+\beta_{0,n}\xi^{n-1}, \\
\eta_{p,\beta}^{n} &= \beta_{2,n}\eta_{p}^{n+1}+\beta_{1,n}\eta_{p}^{n}+\beta_{0,n}\eta_{p}^{n-1}, \ \ \ \ \xi^{n}_{p,\beta} = \beta_{2,n}\xi_{p}^{n+1}+\beta_{1,n}\xi_{p}^{n}+\beta_{0,n}\xi_{p}^{n-1}.
\end{align*}
By the definition of discrete divergence free space $\V_{h}$ and the definition of projection operator $P_{h}$, we have
\begin{align} \label{eq:DLNSDprojterm}
b(\vl_h,\eta_{p,\beta}^{n}) = 0 \ \ \ \text{and} \ \ \ a(\xi_{\beta}^{n},\vl_h)+b(\vl_h,\xi^{n}_{p,\beta})=0.
\end{align}
Choosing $\vl_h = \eta_{\beta}^{n}$ in \eqref{eq:DLNSDError3}, we apply \eqref{eq:DLNSDprojterm} and the Lemma \ref{lemma:GstabDLN} to the equation \eqref{eq:DLNSDError3} to obtain
\begin{align} \label{eq:DLNSDError4}
\begin{Vmatrix} \eta^{n+1} \\ \eta^n \end{Vmatrix}^{2}_{G(\theta)}-	\begin{Vmatrix} \eta^{n} \\ \eta^{n-1} \end{Vmatrix}^{2}_{G(\theta)}+ \left\Vert \sum_{j=0}^{2} {\lambda_{j,n}} \eta^{n-1+j} \right\Vert ^{2}_0+C_{2}(\alpha_{2}k_n - \alpha_{0}k_{n-1}) \left\Vert \eta_{\beta}^{n} \right\Vert_{\U}^2  \notag\\
\leq(\alpha_2\xi^{n+1}+\alpha_1\xi^{n}+\alpha_0\xi^{n-1},\eta_{\beta}^{n})
- \left(\alpha_2k_n-\alpha_0 k_{n-1} \right)\tau(\ul(t_{n,\beta}),p(t_{n,\beta}),\eta_{\beta}^{n}).
\end{align}
Using the Taylor theorem with integral reminder, we have
\begin{align} \label{eq:DLNSDTaylor}
\ul^{n}=\ul^{n+1}+\int_{t_{n+1}}^{t_{n}}\ul_{t}dt, \ \ \ \text{and} \ \ \
\ul^{n-1}=\ul^{n+1}+\int_{t_{n+1}}^{t_{n-1}}\ul_{t}dt.
\end{align}
By \eqref{eq:DLNSDTaylor} and the fact that $\alpha_{2}+\alpha_{1}+\alpha_{0} = 0$,
\begin{align}
\left\Vert \alpha_2\xi^{n+1}+\alpha_1\xi^{n}+\alpha_0\xi^{n-1} \right\Vert_{0}
&= \left\Vert  \alpha_1\int_{t_{n+1}}^{t_{n}}(P_h^{\ul}-Id)\ul_{t}dt+\alpha_0\int_{t_{n+1}}^{t_{n-1}}(P_h^{\ul}-Id)\ul_{t}dt \right\Vert_{0} \notag \\
&\leq C(\theta) \int_{t_{n-1}}^{t_{n+1}} \left\Vert (P_h^{\ul}-Id)\ul_{t} \right\Vert_{0} dt, \label{eq:DLNSDEsti1}
\end{align}
where $Id$ is the identity mapping. Thus by \eqref{eq:EquiNorm}, \eqref{eq:DLNSDEsti1}, Cauchy Schwarz inequality and Young's inequality,
\begin{align} \label{eq:DLNSDEsti2}
\left< \alpha_2\xi^{n+1}+\alpha_1\xi^{n}+\alpha_0\xi^{n-1},\eta_{\beta}^{n} \right>_{0}
\leq C(\theta) \int_{t_{n-1}}^{t_{n+1}} \left\Vert (P_h^{\ul}-Id)\ul_{t} \right\Vert_{0}^2 dt+\frac{C_{2}(\alpha_{2}k_n - \alpha_{0}k_{n-1})}{2} \left\Vert \eta_{\beta}^{n} \right\Vert_{\U}^2.
\end{align}
Summing over \eqref{eq:DLNSDError4} from $n=2,...,M$ ($2 \leq M \leq N-1$) and using \eqref{eq:DLNSDEsti2},
\begin{gather}
\begin{Vmatrix} \eta^{M+1} \\ \eta^{M} \end{Vmatrix}^{2}_{G(\theta)}-\begin{Vmatrix} \eta^{1} \\ \eta^{0} \end{Vmatrix}^{2}_{G(\theta)}
+ \sum_{n=1}^{M} \left\Vert \sum_{j=0}^{2} \lambda_{j,n} \eta^{n-1+j}\right\Vert ^{2}_0
+\sum_{n=1}^{M}\frac{C_{2}(\alpha_{2}k_n - \alpha_{0}k_{n-1})}{2} \left\Vert \eta_{\beta}^{n} \right\Vert_U^2    \notag \\
\leq C(\theta) \sum_{n=1}^{M} \int_{t_{n-1}}^{t_{n+1}}\Vert(P_h^{\ul}-I)\ul_{t}\Vert_{0}^2dt - \sum_{n=1}^{M}(\alpha_2k_n-\alpha_0 k_{n-1})\tau \left(\ul(t_{n,\beta}),p(t_{n,\beta}),\eta_{\beta}^{n}\right). \label{eq:DLNSDError5}
\end{gather}
Then we deal with four terms of $\tau \left(\ul(t_{n,\beta}),p(t_{n,\beta}),\eta_{\beta}^{n}\right)$ respectively. Combining \eqref{eq:EquiNorm}, \eqref{eq:aEstimator}, Lemma \ref{lemma:2ndDLN}
and using Cauchy Schwarz inequality, Young's inequality again, we obtain
\begin{align*}
\left<\frac{\alpha_2\ul^{n+1}+\alpha_1\ul^{n}+\alpha_0\ul^{n-1}}{\alpha_2k_n-\alpha_0k_{n-1}}-\frac{\partial\ul}{\partial t}(t_{n,\beta}),\eta_{\beta}^{n+1} \right>_{0}
\leq C(\theta)(k_n+k_{n-1})^3 \int_{t_{n-1}}^{t_{n+1}} \left\Vert \ul_{ttt} \right\Vert_{0}^2 dt + \frac{C_{2}}{16} \left\Vert \eta_{\beta}^{n} \right\Vert_{\U}^2,
\end{align*}
\begin{align*}
a(\ul_{\beta}^{n}-\ul(t_{n,\beta}),\eta_{\beta}^{n}) \leq C_{1} \left\Vert \ul_{\beta}^{n}-\ul(t_{n,\beta}) \right\Vert_{\U} \left\Vert \eta_{\beta}^{n} \right\Vert_{\U}
\leq C(k_n+k_{n-1})^3 \int_{t_{n-1}}^{t_{n+1}} \left\Vert \ul_{tt} \right\Vert_{\U}^2 dt + \frac{C_{2}}{16}\left\Vert \eta_{\beta}^{n} \right\Vert_{\U}^2,
\end{align*}
\begin{align*}
b(\eta_{\beta}^{n+1}, p^{n}_{\beta}-p(t_{n,\beta})) \leq C \left\Vert p^{n+1}_{\beta}-p(t_{n,\beta}) \right\Vert \left\Vert \eta_{\beta}^{n} \right\Vert_{\U}
\leq C(k_n+k_{n-1})^3 \int_{t_{n-1}}^{t_{n+1}} \left\Vert p_{tt} \right\Vert^{2} dt + \frac{C_{2}}{16}\left\Vert \eta_{\beta}^{n} \right\Vert_{\U}^2,
\end{align*}
\begin{align}
\ \ \ \ \ \ \left< \F^{n}_{\beta}-\F(t_{n,\beta}),\eta_{\beta}^{n} \right>_{\U'} \leq \left\Vert \F^{n}_{\beta}-\F(t_{n,\beta}) \right\Vert_{\U'} \left\Vert \eta_{\beta}^{n} \right\Vert_{\U}
\leq C (k_n+k_{n-1})^3 \int_{t_{n-1}}^{t_{n+1}} \left\Vert \F_{tt} \right\Vert_{\U'}^2 dt + \frac{C_{2}}{16}\left\Vert \eta_{\beta}^{n} \right\Vert_{\U}^2. \label{eq:tauEstifour}
\end{align}
Since $\eta^{1}=0 = \eta^{0}$ and
by \eqref{eq:ProjProperty}, \eqref{eq:EstAvestep}, the definition of $G(\theta)$-norm in \eqref{eq:defGnorm}, estimators in \eqref{eq:tauEstifour}, \eqref{eq:DLNSDError5} becomes
\begin{gather}
\frac{1+\theta}{4} \left\Vert \eta^{M+1} \right\Vert^2_0 + \frac{1-\theta}{4} \left\Vert \eta^{M} \right\Vert^2_0 + \sum_{n=1}^{M} \left\Vert \sum_{j=0}^{2} \lambda_{j,n} \eta^{n-1+j} \right\Vert_{0}^{2}
+ \frac{C_{2}}{4} \sum_{n=1}^{M} \left(\alpha_{2}k_{n} - \alpha_{0}k_{n-1} \right) \left\Vert \eta_{\beta}^{n} \right\Vert_{\U}^2   \notag \\
\leq C(\theta) \max_{1\leq n\leq N-1}\left\{( k_{n} + k_{n-1} )^{4} \right\}\left( \left\Vert p_{tt} \right\Vert^{2}_{2,L^{2}} + \left\Vert \ul_{ttt} \right\Vert_{2,0}^2 + \left\Vert \ul_{tt} \right\Vert_{2,\U}^2 + \left\Vert \F_{tt} \right\Vert_{2,\U'}^2  \right)
+\sum_{n=1}^{N-1} C(\theta) \int_{t_{n-1}}^{t_{n+1}} \left\Vert (P_h^{\ul}-I)\ul_{t} \right\Vert_{0}^2dt \notag \\
\leq  C(\theta) \max_{1\leq n\leq N-1}\left\{ ( k_{n} + k_{n-1} )^{4} \right\} \left( \left\Vert p_{tt} \right\Vert^{2}_{2,L^{2}} + \left\Vert \ul_{ttt} \right\Vert_{2,0}^2 + \left\Vert \ul_{tt} \right\Vert_{2,\U}^2 + \left\Vert \F_{tt} \right\Vert_{2,\U'}^2  \right) + C(\theta) h^{4} \left\Vert \ul_{t} \right\Vert_{2,2}^2.  \label{eq:DLNSDetaEsti}
\end{gather}
Using triangle inequality,
\begin{align} \label{eq:DLNSDErrEstL2Tri}
\left\Vert \left| \e \right| \right\Vert_{2,0} \leq \left\Vert \left| \xi \right| \right\Vert_{2,0} + \left\Vert \left| \eta \right| \right\Vert_{2,0},
\end{align}
By \eqref{eq:ProjProperty} and \eqref{eq:DLNSDetaEsti}, we have
\begin{align} \label{eq:DLNSDErrEstL2xi}
\left\Vert \left| \xi \right| \right\Vert_{2,0} = \left( \sum_{n=0}^{N-1} k_{n} \left\Vert \xi^{n+1} \right\Vert_{0}^{2} \right)^{1/2}
= \left( \sum_{n=0}^{N-1} k_{n} \left\Vert \ul^{n} - P_{h}^{\ul} \ul^{n} \right\Vert_{0}^{2} \right)^{1/2} \leq \left( \sum_{n=0}^{N-1} C_{3}h^{4} k_{n} \left\Vert \ul^{n+1} \right\Vert_{2}^{2} \right)^{1/2}
\leq C h^{2} \left\Vert \left| \ul \right| \right\Vert_{2,2}.
\end{align}
\begin{align}
 \left\Vert \left| \eta \right| \right\Vert_{2,0}
\leq & C(\theta) \left( \sum_{n=0}^{N-1} k_{n} \right)^{1/2}
\bigg\{ \max_{1\leq n\leq N-1}\left\{ ( k_{n} + k_{n-1} )^{2} \right\} \left( \left\Vert p_{tt} \right\Vert_{2,L^{2}} + \left\Vert \ul_{ttt} \right\Vert_{2,0} + \left\Vert \ul_{tt} \right\Vert_{2,\U} + \left\Vert \F_{tt} \right\Vert_{2,\U'} \right) \notag \\
& \ \ \ \ \ \ \ \ \ \ \ \ \ \ \ \ \ \ \ \ \ \ \ \ \ \ \ \ \ \ \ \ \ \ \ \ + h^{2} \left\Vert \ul_{t} \right\Vert_{2,2} \bigg\} \notag \\
\leq & C(\theta) \sqrt{T}
\left\{ \max_{1\leq n\leq N-1}\left\{ ( k_{n} + k_{n-1} )^{2} \right\} \left( \left\Vert p_{tt} \right\Vert_{2,L^{2}} + \left\Vert \ul_{ttt} \right\Vert_{2,0} + \left\Vert \ul_{tt} \right\Vert_{2,\U} + \left\Vert \F_{tt} \right\Vert_{2,\U'} \right) + h^{2} \left\Vert \ul_{t} \right\Vert_{2,2} \right\} \label{eq:DLNSDErrEstL2eta}
\end{align}
Combining \eqref{eq:DLNSDErrEstL2Tri}, \eqref{eq:DLNSDErrEstL2xi} and \eqref{eq:DLNSDErrEstL2eta} results in \eqref{eq:DLNSDErrEstL2}.
For second part, we have
\begin{gather}
\sum_{n=1}^{N-1} \left(\alpha_{2}k_{n} - \alpha_{0}k_{n-1}\right) \left\Vert \ul(t_{n,\beta}) - \ul_{h,\beta}^{n} \right\Vert_{\U}^{2} \notag\\
\leq  C(\theta)\sum_{n=1}^{N-1} \left(k_{n} + k_{n-1}\right) \left\Vert \ul(t_{n,\beta}) - \ul_{\beta}^{n} \right\Vert_{\U}^{2}
+ \sum_{n=1}^{N-1} \left(\alpha_{2}k_{n} - \alpha_{0} k_{n-1}\right) \left\Vert \ul_{\beta}^{n} - \ul_{h,\beta}^{n} \right\Vert_{\U}^{2}.  \label{eq:DLNSDErrEstH1}
\end{gather}
Using Lemma \ref{lemma:2ndDLN},
\begin{align*}
C(\theta)\sum_{n=1}^{N-1} \left(k_{n} + k_{n-1}\right) \left\Vert \ul(t_{n,\beta}) - \ul_{\beta}^{n} \right\Vert_{\U}^{2}
\leq C(\theta) \max_{1\leq n\leq N-1}\left\{ ( k_{n} + k_{n-1} )^{4} \right\} \left\Vert \ul_{tt} \right\Vert_{2,\U}^{2}.
\end{align*}
And
\begin{align} \label{eq:DLNSDErrEstH1term2}
\sum_{n=1}^{N-1} \left(\alpha_{2}k_{n} - \alpha_{0} k_{n-1}\right) \left\Vert \ul_{\beta}^{n} - \ul_{h,\beta}^{n} \right\Vert_{\U}^{2}
\leq C(\theta) \sum_{n=1}^{N-1} \left(k_{n} + k_{n-1}\right) \left\Vert \xi_{\beta}^{n} \right\Vert_{\U}^{2} + \sum_{n=1}^{N-1} \left(\alpha_{2}k_{n}-\alpha_{0}k_{n-1}\right) \left\Vert \eta_{\beta}^{n} \right\Vert_{\U}^{2}
\end{align}
By \eqref{eq:ProjProperty} and linearity of the projection operator $P_{h}$,
\begin{align} \label{eq:DLNSDErrEstH1term21}
\left\Vert \xi_{\beta}^{n} \right\Vert_{\U}^{2} = \left\Vert P_{h}^{\ul} \ul_{\beta}^{n} - \ul_{\beta}^{n}  \right\Vert_{\U}^{2}
\leq C h^{2} \left\Vert\ul_{\beta}^{n}  \right\Vert_{2}^{2} \leq C h^{2} \left\Vert\ul_{\beta}^{n} - \ul(t_{n,\beta}) \right\Vert_{2}^{2} + C h^{2} \left\Vert \ul(t_{n,\beta}) \right\Vert_{2}^{2}
\end{align}
Applying Lemma \ref{lemma:2ndDLN} again to \eqref{eq:DLNSDErrEstH1term21},
\begin{align} \label{eq:DLNSDErrEstH1term22}
C(\theta) \sum_{n=1}^{N-1} \left(k_{n} + k_{n-1}\right) \left\Vert \xi_{\beta}^{n} \right\Vert_{\U}^{2}
\leq C(\theta) h^{2} \max_{1\leq n\leq N-1}\left\{ ( k_{n} + k_{n-1} )^{4} \right\} \left\Vert \ul_{tt} \right\Vert_{2,2}^{2}
+ C(\theta) h^{2} \left\Vert \left| \ul_{\beta} \right| \right\Vert_{2,2}^{2}.
\end{align}
Combining \eqref{eq:DLNSDetaEsti}, \eqref{eq:DLNSDErrEstH1}, \eqref{eq:DLNSDErrEstH1term2} and \eqref{eq:DLNSDErrEstH1term22}, we obtain
\begin{gather*}
\sum_{n=1}^{N-1} \left(\alpha_{2}k_{n} - \alpha_{0}k_{n-1}\right) \left\Vert \ul(t_{n,\beta}) - \ul_{h,\beta}^{n} \right\Vert_{\U}^{2} \\
\leq  C(\theta) \max_{1\leq n\leq N-1}\left\{ ( k_{n} + k_{n-1} )^{4} \right\} \left( \left\Vert p_{tt} \right\Vert^{2}_{2,L^{2}} + \left\Vert \ul_{ttt} \right\Vert_{2,0}^2 + \left\Vert \ul_{tt} \right\Vert_{2,\U}^2 + \left\Vert \F_{tt} \right\Vert_{2,\U'}^2  \right) + C(\theta) h^{4} \left\Vert \ul_{t} \right\Vert_{2,2}^2 \\
+ C(\theta) h^{2} \max_{1\leq n\leq N-1}\left\{ ( k_{n} + k_{n-1} )^{4} \right\} \left\Vert \ul_{tt} \right\Vert_{2,2}^{2} + C(\theta) h^{2} \left\Vert \left| \ul_{\beta} \right| \right\Vert_{2,2}^{2},
\end{gather*}
which results in \eqref{eq:DLNSDErrEstH1Final}
\end{proof}

\section{Numerical Tests} \label{sec:numtests}
In this section, we use two numerical experiments to verify two distinct properties of the DLN algorithm (stability and consistency). Both numerical tests are implemented by FreeFEM++. The first test confirms
that the variable timestepping DLN algorithm is stable for different values of parameter $\theta \in [0,1]$. In the second experiment, we apply the constant timestepping DLN algorithm to check the second order convergence of the approximate solutions as well as compare it with BDF2 scheme.
\subsection{Test of Variable Timestepping DLN algorithm}
In this experiment, we use the example mentioned in \cite{cao2014SDparallelNoniterative,jiang2019SDefficientEnsemble}. Considering the model problem on $\Omega_f=[0,\pi]\times[0,1]$ and $\Omega_p=[0,\pi]\times[-1,0]$ with the interface $\Gamma=[0,\pi]\times [0]$:
\begin{align*}
\u_{f} &=\left(\frac{1}{\pi}\sin(2\pi y)\cos(x)e^t, \left(-2+\frac{1}{\pi^2}\sin(\pi y)^2 \right)\sin(x)e^t \right),\\
p &=0,\\
\phi &=(e^{y}-e^{-y})\sin(x)e^t.
\end{align*}
For this test, we set the physical parameters $\rho$, g, $\nu$, $\mathbf{K}$, $S_0$ and $\mu_{BJS}$ all equal to 1 and we consider the cases of parameters $\theta=0.2,0.5,0.7$ in DLN scheme. The initial conditions, boundary conditions and the source terms follow from the exact solution. We use the well-know Taylor-Hood element (P2-P1)  for the fluid equation and the piecewise quadratic polynomials (P2) for the porous equation. To see the effect on the results by change of time steps, we fix the diameter $h=100$ for space triangulation. We apply the DLN algorithm to this test problem for 40 time steps and introduce the timestep function similar to that in \cite{chen2019VariArti}:
\begin{align} \label{eq:NumTest1stepfunc}
k_n=\begin{cases}
0.1&\text{$0\leq n\leq 10$},\\
0.1+0.05\sin(10t_n)&\text{$n>10$}.
\end{cases}
\end{align}
The graph of the time step function \eqref{eq:NumTest1stepfunc} is given in Figure \ref{fig2}.
\begin{figure}[ptbh]
\centering
\par
{
		\begin{minipage}[t]{0.45\linewidth}
			\centering
			\includegraphics[width=2.5in,height=2.0in]{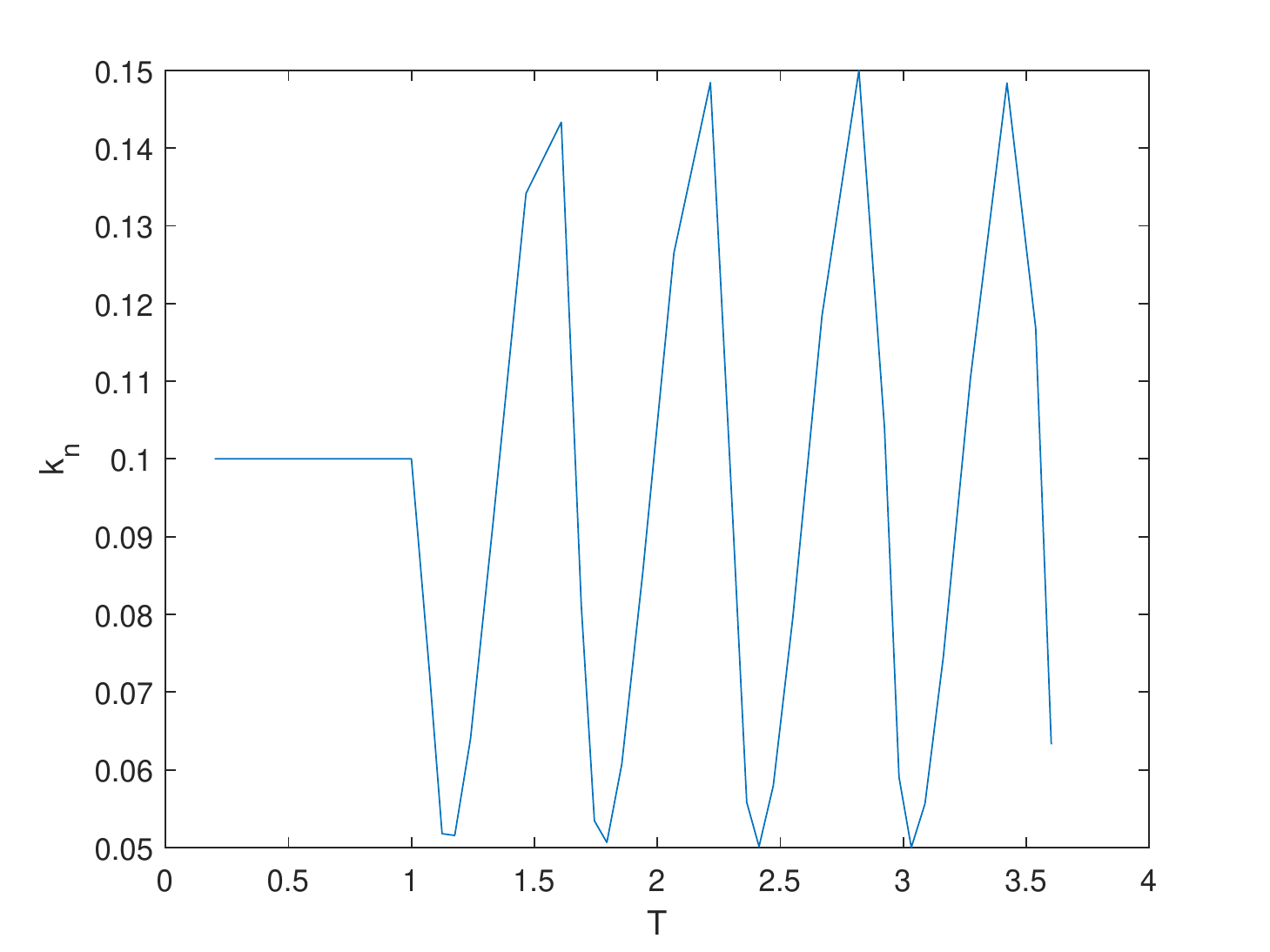}\\
			\vspace{0.02cm}
					\end{minipage}	} \centering
\vspace{-0.2cm}
\caption{Change of step size $k_{n}$.}
\label{fig2}
\end{figure}
Figure \ref{fig3} shows the speed contours and velocity streamlines with parameter $\theta=0.2,0.5,0.7$ respectively. From the graphs in Figure \ref{fig3}, we observe that good performance can be obtained for all three cases. Figure \ref{fig4a} and Figure \ref{fig4b} respectively show the comparison between the approximate solutions and the true solutions of the incompressible fluid velocity $\u_{f}$ and porous media fluid hydraulic head $\phi$  with different $\theta$. The variable timestepping DLN algorithm approximate exact solutions well, which confirms the stability of the DLN algorithm.
\begin{figure}[ptbh]
\centering
\par
\subfigure[$\theta = 0.2$]{
\begin{minipage}[t]{0.45\linewidth}
\centering
\includegraphics[width=2.5in,height=2.0in]{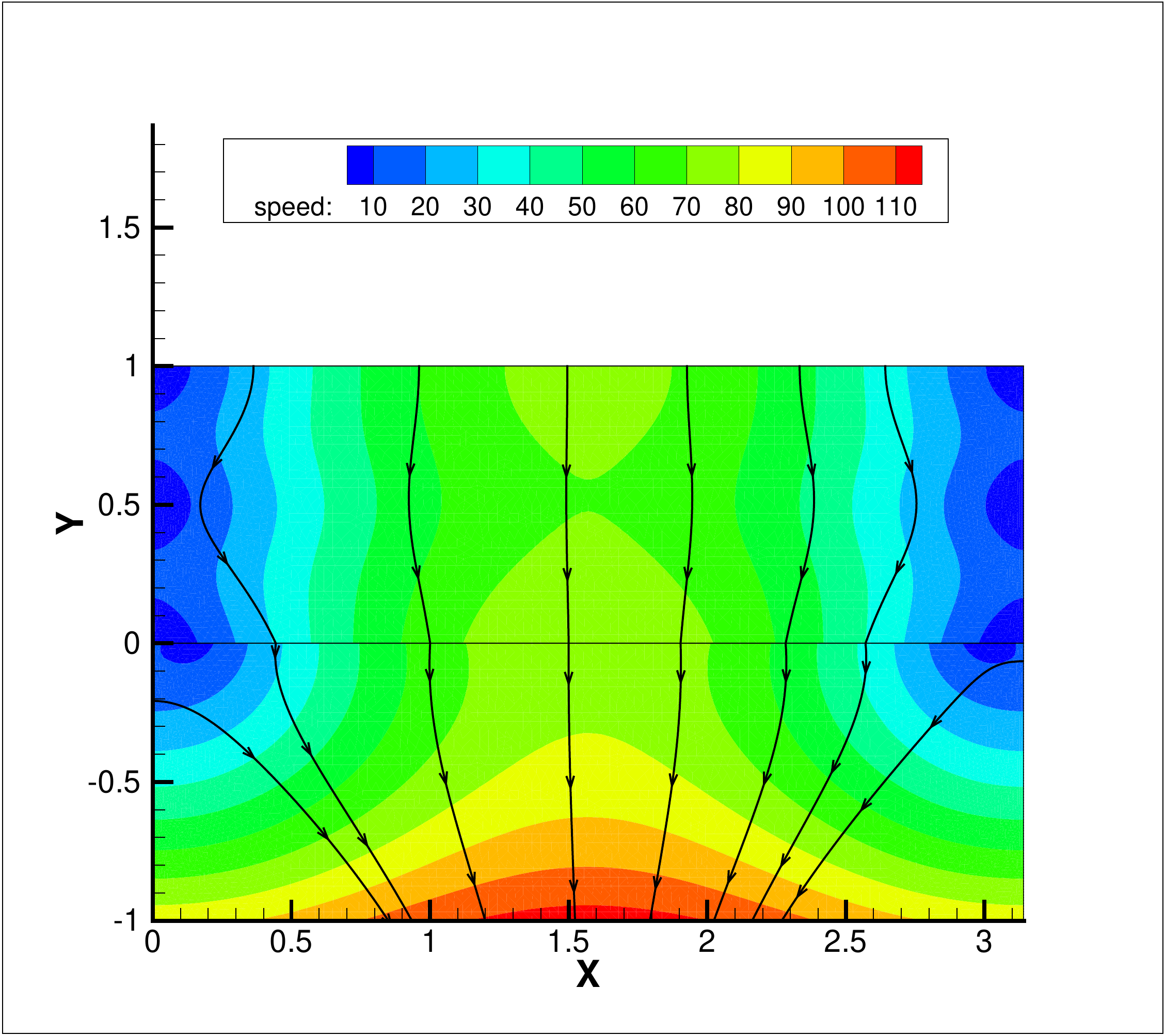}\\
\vspace{0.02cm}
\end{minipage}}%
\subfigure[$\theta =0.5$]{
\begin{minipage}[t]{0.33\linewidth}
\centering
\includegraphics[width=2.5in,height=2.0in]{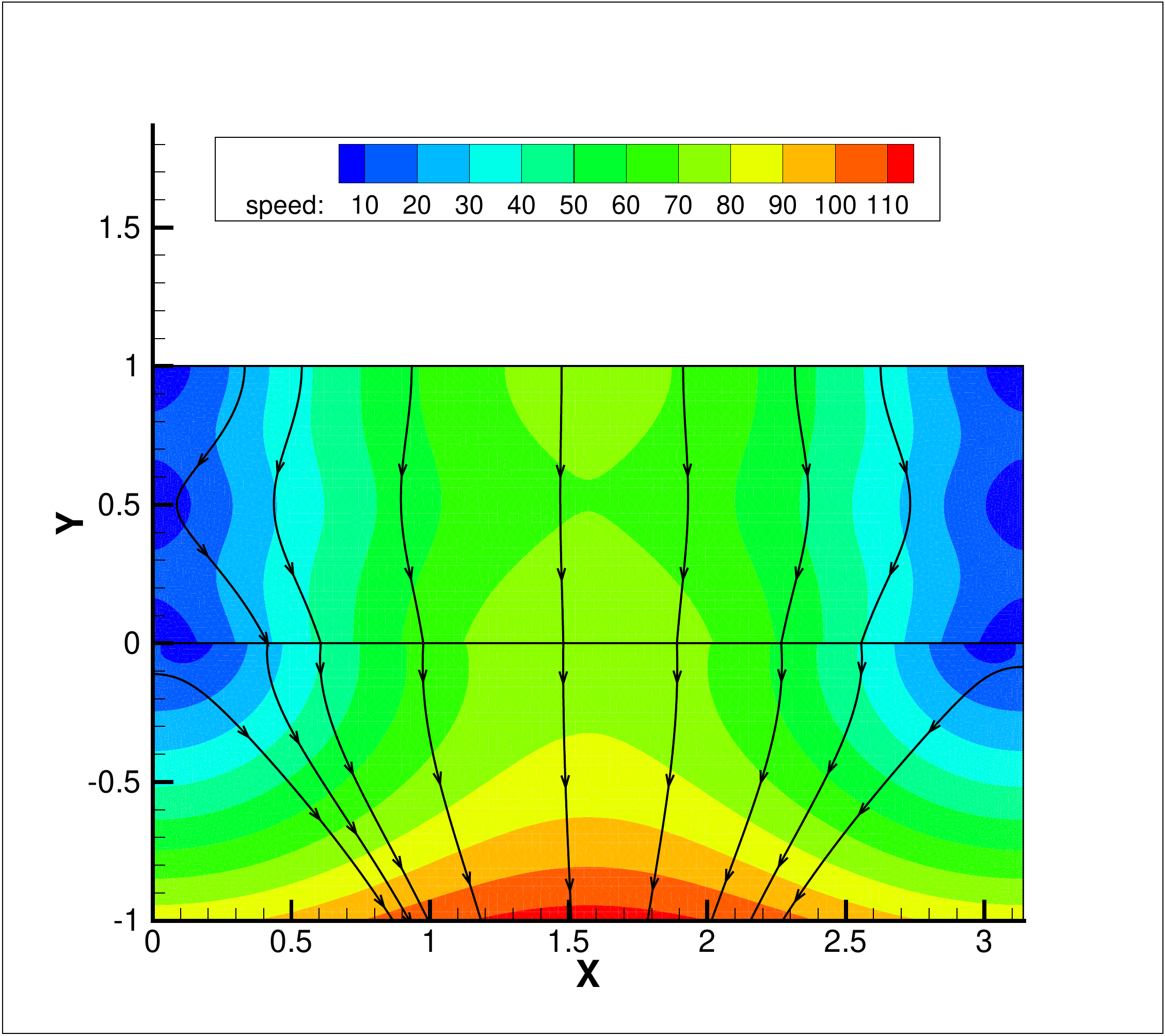}\\
\vspace{0.02cm}
\end{minipage}}\newline
\subfigure[$\theta =0.7$]{
\begin{minipage}[t]{0.33\linewidth}
\centering
\includegraphics[width=2.5in,height=2.0in]{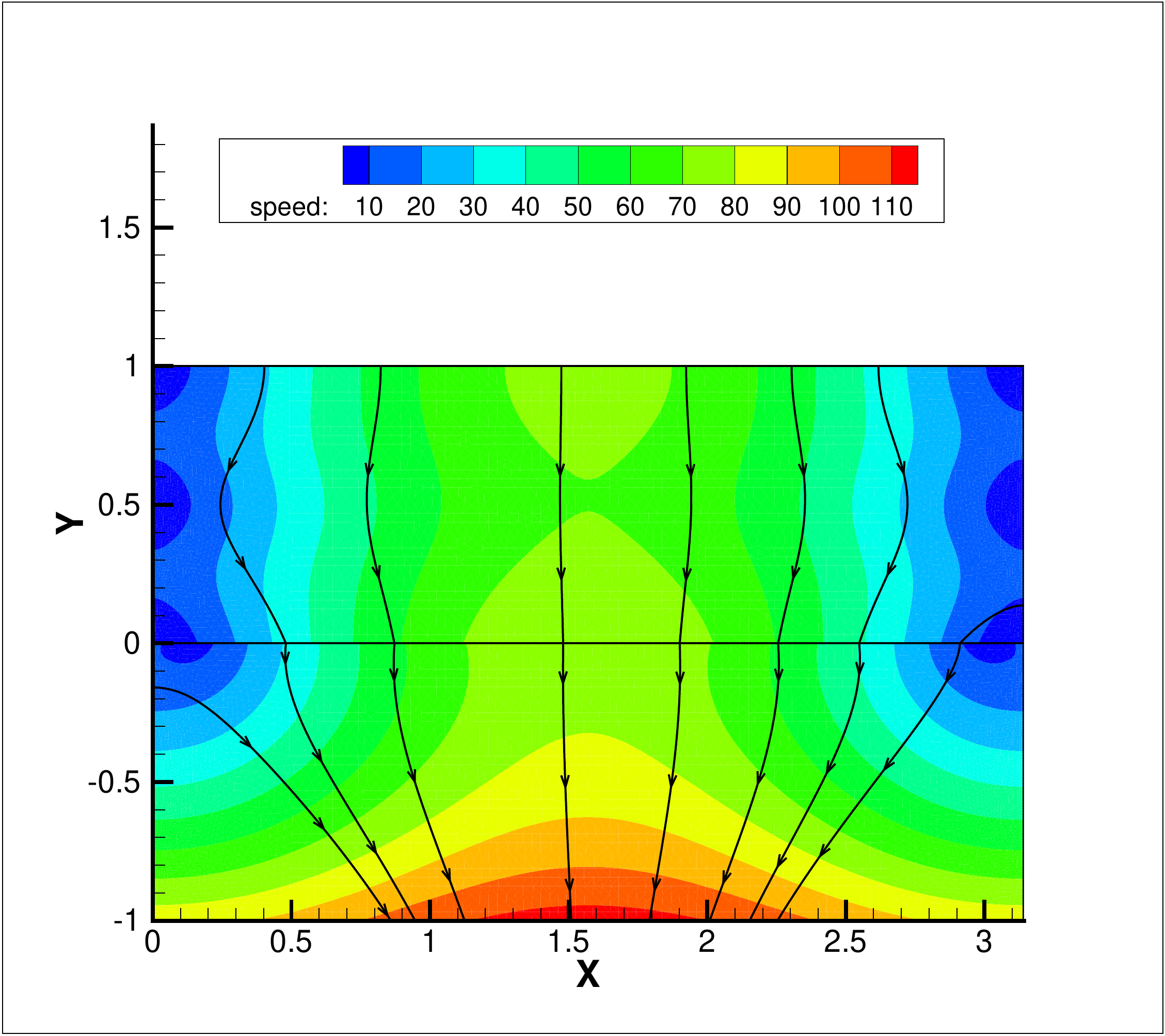}\\
\vspace{0.02cm}
\end{minipage}}%
\par
\centering
\vspace{-0.2cm}
\caption{The speed contours and velocity streamlines with $\theta=0.2,0.5,0.7$. }
\label{fig3}
\end{figure}

\begin{figure}[ptbh]
\centering
\par
\subfigure[Comparison for velocity $\u_{f}$]{ \label{fig4a}
\begin{minipage}[t]{0.45\linewidth}
\centering
\includegraphics[width=2.8in,height=2.2in]{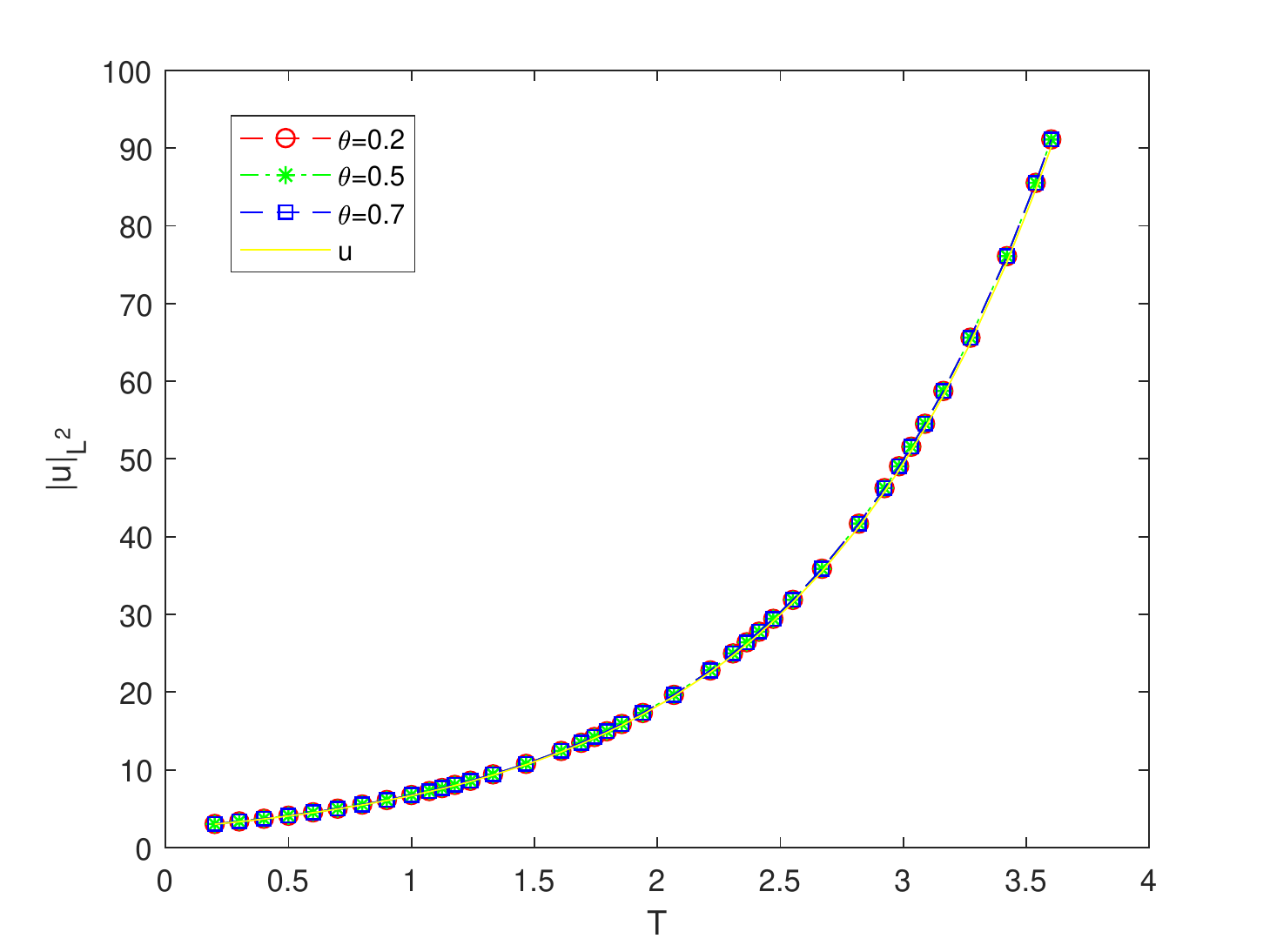}\\
\vspace{0.02cm}
\end{minipage}} \label{fig4}
\subfigure[Comparison for hydraulic head $\phi$]{ \label{fig4b}
\begin{minipage}[t]{0.45\linewidth}
\centering
\includegraphics[width=2.8in,height=2.2in]{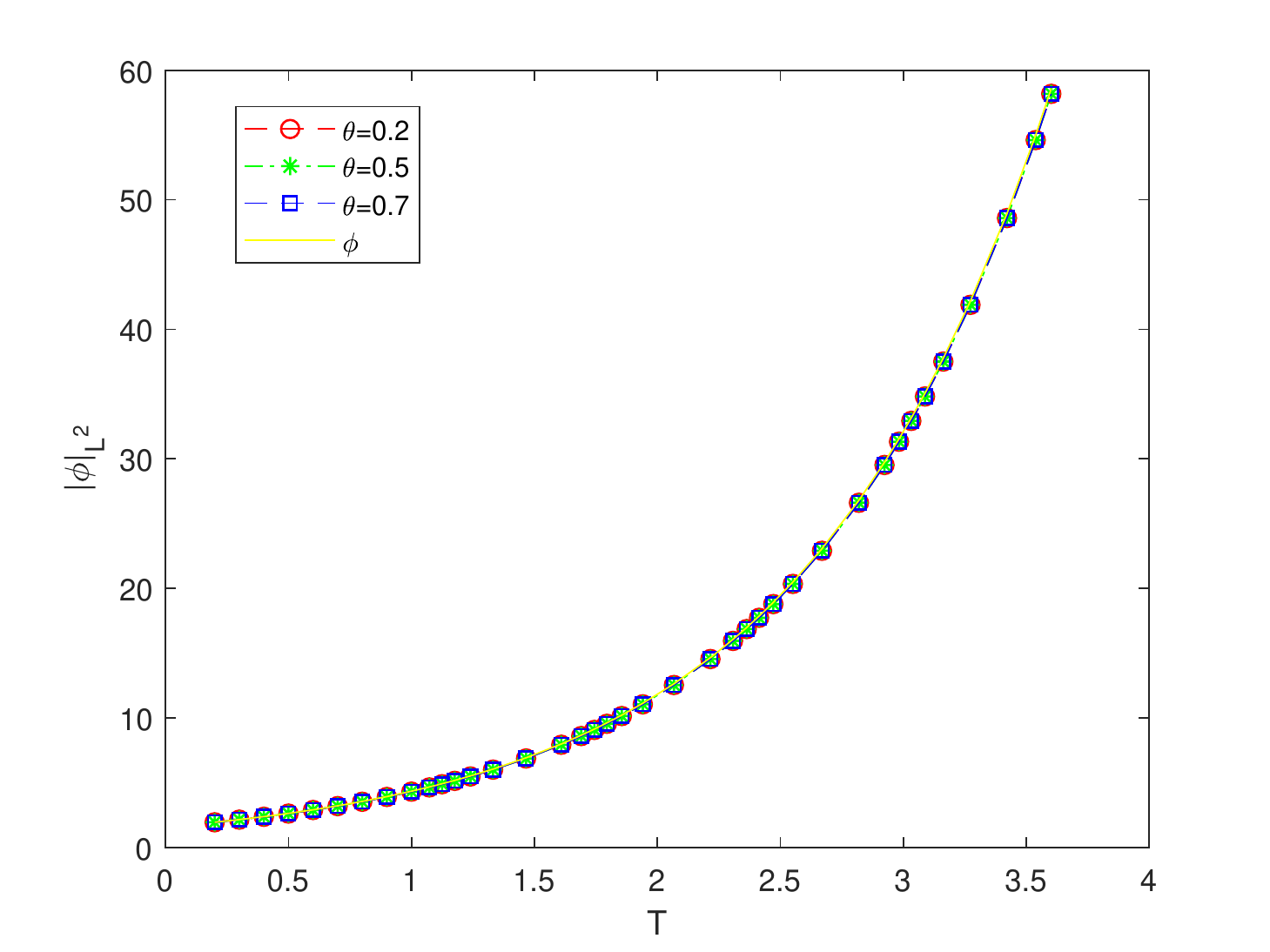}\\
\vspace{0.02cm}
\end{minipage}}
\par
\centering
\vspace{-0.2cm}
\caption{Comparison between the approximate solutions and the exact solutions with different parameter $\theta$.}
\end{figure}

\subsection{Test of constant Timestepping DLN algorithm}
For constant timestep test, we refer to the numerical example in \cite{mu2010SDdecoupledmixed}.
Let the computational domain $\Omega$ be composed of $\Omega_f=(0,1)\times(1,2)$ and $\Omega_p=(0,1)\times(0,1)$ with the interface $\Gamma=(0,1)\times \{1\}$. We set the total time $T=1$. The exact solution is:
\begin{align*}
\u_{f}&= \left((x^2(y-1)^2+y)\cos(t),-\frac{2}{3}x(y-1)^3\cos(t)+(2-\pi\sin(\pi x))\cos(t) \right), \\
p&=\left( 2-\pi\sin(\pi x) \right)\sin \left(\frac{1}{2}\pi y \right)\cos(t), \\
\phi&=\left(2-\pi\sin(\pi x) \right) \left( 1-y-\cos(\pi y) \right) \cos(t).
\end{align*}
For this test, MINI (P1b-P1) space and piecewise linear polynomials (P1) space are used for the approximation of the incompression fluid and the porous equation respectively. To confirm the consistency of the DLN algorithm, we set $h=\Delta t$ and calculate the errors and convergence rates for the functions $\u_{f}$, $\phi$ and $p$. The rate of convergence $r$ is calculated by
\begin{align*}
r=\ln(e(\Delta t_{1})/e(\Delta t_{2}))/\ln(\Delta t_{1} / \Delta t_{2}),
\end{align*}
where $e(\Delta t)$ is the error computed by the DLN algorithm with time stepsize $\Delta t$.

Table \ref{table1}, \ref{table2} and \ref{table3} show the fluid velocity $\u_{f}$, hydraulic head $\phi$  and pressure $p$ errors of the DLN algorithm when $\theta=0.2,0.5,0.7$. The results are almost the same for three different $\theta$, but as $\theta$ increases, the errors of $\u_{f}$ decrease slightly, while the errors of $\phi$ increase. Thus how to choose the best parameters leaves an open question.
Moreover Table \ref{table4}, Table \ref{table5} and Table \ref{table6} show the convergence rate of velocity $\u_{f}$, hydraulic head $\phi$ and pressure $p$ with different $\theta$ and therefore verify the second-order convergence of the DLN algorithm. Finally, Table \ref{table7} shows the corresponding errors obtained by the common BDF2 method. By comparison, we can see that the DLN algorithm obtains a better hydraulic head $\phi$ than BDF2 method.

\begin{table}[H]
	\centering
	\caption{The errors for DLN scheme with $\theta=0.2$.}
	\begin{tabular}{llllll}
		\hline
		$\Delta t=h$     &$\Vert |\e_{\u_{f}}| \Vert_{2,0}$ & $\Vert |\e_{\u_{f}} |\Vert_{2,1}$   & $\Vert |\e_{\phi} | \Vert_{2,0}$  & $\Vert |\e_{\phi}| \Vert_{2,1}$& $\Vert |e_{p}| \Vert_{2,0}$\\ \hline
		$1/10$              &0.0163655             &0.599657                 &0.0143625                 &0.552125                &0.175753 \\
		$1/16$              &0.00657067           &0.354318                 &0.00587243               &0.359717                &0.0785158    \\
		
		$1/22$             &0.00353871            &0.255182                 &0.00317754               &0.268333                &0.0490189   \\
		
		$1/28$             &0.00218857           &0.191492	                &0.00198363               &0.2117                     &0.0306542     \\
		$1/34$             &0.00150194           &0.160602	                &0.00135819               &0.177254                  &0.0213342     \\
		\hline
	\end{tabular} \label{table1}
\end{table}

\begin{table}[H]
	\centering
	\caption{The errors for DLN scheme with $\theta=0.5$.}
	\begin{tabular}{llllll}
		\hline
		$\Delta t=h$     &$\Vert |\e_{\u_{f}}| \Vert_{2,0}$ & $\Vert |\e_{\u_{f}} |\Vert_{2,1}$   & $\Vert |\e_{\phi} | \Vert_{2,0}$  & $\Vert |\e_{\phi}| \Vert_{2,1}$& $\Vert |e_{p}| \Vert_{2,0}$\\ \hline
		$1/10$              &0.01615                 &0.506002                 &0.0146238                 &0.551755                 &0.138243  \\
		$1/16$              &0.00652393          &0.311263                  &0.00599802               &0.359655                &0.0637115    \\
		
		$1/22$             &0.00351853           &0.22917                    &0.00324735               &0.268314                &0.04083   \\
		
		$1/28$             &0.00218086           &0.176397	                &0.00202875               &0.211693                 &0.0260884     \\
		$1/34$             &0.00149633           &0.148517	                &0.0013883                  &0.177249                 &0.0184629     \\
		\hline
	\end{tabular} \label{table2}
\end{table}

\begin{table}[H]
	\centering
	\caption{The errors for DLN scheme with $\theta=0.7$.}
	\begin{tabular}{llllll}
		\hline
		$\Delta t=h$     &$\Vert |\e_{\u_{f}}| \Vert_{2,0}$ & $\Vert |\e_{\u_{f}} |\Vert_{2,1}$   & $\Vert |\e_{\phi} | \Vert_{2,0}$  & $\Vert |\e_{\phi}| \Vert_{2,1}$& $\Vert |e_{p}| \Vert_{2,0}$\\ \hline
		$1/10$              &0.0161161               &0.488013                 &0.0150263                 &0.551591                 &0.128276 \\
		$1/16$              &0.00652022          &0.30443                   &0.00616699               &0.359622                &0.0604363    \\
		
		$1/22$             &0.00351759           &0.225303                  &0.00333733               &0.268301                &0.0393132   \\
		
		$1/28$             &0.00218125           &0.174198	                &0.00208573                &0.211687                 &0.0252779     \\
		$1/34$             &0.00149674           &0.14679	                &0.00142616                  &0.177246                 &0.0179642     \\
		\hline
	\end{tabular} \label{table3}
\end{table}

\begin{table}[H]
	\centering
	\caption{The convergence order of errors for DLN scheme with $\theta=0.2$.}
	\begin{tabular}{llllll}
		\hline
		$\Delta t=h$     & $r_{\u_{f},0}$ & $r_{\u_{f},1}$   &$r_{\phi,0}$  &$r_{\phi,1}$    &$r_{p,0}$\\ \hline
		$1/10$              &-                &-                 &-                    &-                    &- \\
		$1/16$              &1.9416       &1.11949       &1.90286         &0.911604       &1.71441    \\
		
		$1/22$             &1.94331      &1.03066       &1.92857          &0.920353       &1.47931  \\
		
		$1/28$             &1.99249       &1.19062	      &1.95378         &0.982976       &1.94657     \\
		$1/34$             &1.93911       &0.906045	     &1.9509            &0.914669        &1.86683     \\
		\hline
	\end{tabular} \label{table4}
\end{table}

\begin{table}[H]
	\centering
	\caption{The convergence order of errors for DLN scheme with $\theta=0.5$.}
	\begin{tabular}{llllll}
		\hline
		$\Delta t=h$     & $r_{\u_{f},0}$ & $r_{\u_{f},1}$   &$r_{\phi,0}$  &$r_{\phi,1}$    &$r_{p,0}$\\ \hline
	    $1/10$              &-                &-                 &-                    &-                    &- \\
		$1/16$              &1.92895      &1.03382       &1.8962         &0.910541       &1.64818    \\
		
		$1/22$             &1.93885      &0.961447       &1.92678         &0.920035       &1.39722  \\
		
		$1/28$             &1.98342       &1.08527	      &1.95063         &0.982834       &1.85737     \\
		$1/34$             &1.94019       &0.886068	     &1.9538           &0.914617        &1.78066     \\
		\hline
	\end{tabular} \label{table5}
\end{table}

\begin{table}[H]
	\centering
	\caption{The convergence order of errors for DLN scheme with $\theta=0.7$.}
	\begin{tabular}{llllll}
		\hline
		$\Delta t=h$     & $r_{\u_{f},0}$ & $r_{\u_{f},1}$   &$r_{\phi,0}$  &$r_{\phi,1}$    &$r_{p,0}$\\ \hline
	    $1/10$              &-                &-                 &-                    &-                    &- \\
		$1/16$              &1.92532      &1.00404       &1.89485         &0.910111       &1.60126    \\
		
		$1/22$             &1.93791      &0.945172       &1.92819         &0.919886       &1.35037  \\
		
		$1/28$             &1.98157       &1.06674	      &1.94911         &0.982746       &1.83126     \\
		$1/34$             &1.93971       &0.881715	     &1.95786           &0.914595        &1.75915     \\
		\hline
	\end{tabular} \label{table6}
\end{table}

\begin{table}[H]
	\centering
	\caption{The errors for BDF2 scheme.}
	\begin{tabular}{llllll}
		\hline
		$\Delta t=h$     &$\Vert |\e_{\u_{f}}| \Vert_{2,0}$ & $\Vert |\e_{\u_{f}} |\Vert_{2,1}$   & $\Vert |\e_{\phi} | \Vert_{2,0}$  & $\Vert |\e_{\phi}| \Vert_{2,1}$& $\Vert |e_{p}| \Vert_{2,0}$\\ \hline
		$1/10$              &0.0160291               &0.450396                &0.0165148                 &0.551278                &0.116047  \\
		$1/16$              &0.00650765          &0.290462                   &0.00680715               &0.359553                &0.0561277    \\
		
		$1/22$             &0.00351566           &0.2176                      &0.0036845                &0.268273                &0.0373131   \\
		
		$1/28$             &0.00218218           &0.169732                  &0.00230674                &0.211677                 &0.024088     \\
		$1/34$             &0.00149872           &0.143413	                &0.00157485                 &0.177236                &0.0171673     \\
		\hline
	\end{tabular} \label{table7}
\end{table}

\section{Conclusions} \label{sec:conclusions}
This report has shown that the DLN algorithm has advantages on variable timestepping analysis for the unsteady Stokes/Darcy model due to unconditional, long time $G$-stability and second order accuracy under variable time steps.
Stability of the approximate solutions are obtained by $G$-stability of the DLN algorithm and second order accuracy of the numerical simulations are derived from combination of $G$-stability and consistency properties of the DLN
algorithm. Therefore the variable time stepping algorithm would be popular if the complexity of the DLN algorithm is overcome. One efficient way would be implementation of the DLN algorithm through adding time filters on certain first order implicit method. Moreover, adaptivity process for the DLN algorithm would highly reduce the computation cost if reliable estimators of local truncation error can be obtained.

\bibliographystyle{amsplain}

\begin{thebibliography}{10}

\bibitem{arbogast2009SDdiscretizationmultgrid}
Todd Arbogast and Mario San~Martin Gomez, \emph{A discretization and multigrid
  solver for a {D}arcy--{S}tokes system of three dimensional vuggy porous
  media}, Computational Geosciences \textbf{13} (2009), no.~3, 331--348.

\bibitem{badea2010SDmodel}
Lori Badea, Marco Discacciati, and Alfio Quarteroni, \emph{Numerical analysis
  of the {N}avier--{S}tokes/{D}arcy coupling}, Numerische Mathematik
  \textbf{115} (2010), no.~2, 195--227.

\bibitem{cao2014SDparallelNoniterative}
Yanzhao Cao, Max Gunzburger, Xiaoming He, and Xiaoming Wang, \emph{Parallel,
  non-iterative, multi-physics domain decomposition methods for time-dependent
  {S}tokes-{D}arcy systems}, Mathematics of Computation \textbf{83} (2014),
  no.~288, 1617--1644.

\bibitem{ccecsmeliouglu2008SDBE}
A~{\c{C}}e{\c{s}}melio{\u{g}}lu and B{\'e}atrice Rivi{\`e}re, \emph{Analysis of
  time-dependent {N}avier--{S}tokes flow coupled with {D}arcy flow}, Journal of
  Numerical Mathematics \textbf{16} (2008), no.~4, 249--280.

\bibitem{ccecsmeliouglu2009SDGalerkin}
Ay{\c{c}}{\i}l {\c{C}}e{\c{s}}melio{\u{g}}lu and B{\'e}atrice Rivi{\`e}re,
  \emph{Primal discontinuous {G}alerkin methods for time-dependent coupled
  surface and subsurface flow}, Journal of Scientific Computing \textbf{40}
  (2009), no.~1-3, 115--140.

\bibitem{chen2019VariArti}
Robin Chen, William Layton, and Michael McLaughlin, \emph{Analysis of
  variable-step/non-autonomous artificial compression methods}, Journal of
  Mathematical Fluid Mechanics \textbf{21} (2019).

\bibitem{chen2013SDsecondefficient}
Wenbin Chen, Max Gunzburger, Dong Sun, and Xiaoming Wang, \emph{Efficient and
  long-time accurate second-order methods for the {S}tokes--{D}arcy system},
  SIAM Journal on Numerical Analysis \textbf{51} (2013), no.~5, 2563--2584.

\bibitem{chen2016SDthirdefficient}
\bysame, \emph{An efficient and long-time accurate third-order algorithm for
  the {S}tokes--{D}arcy system}, Numerische Mathematik \textbf{134} (2016),
  no.~4, 857--879.

\bibitem{dahlquist1978positive}
G~Dahlquist, \emph{Positive functions and some applications to stability
  questions for numerical methods. recent advances in numerical analysis},
  Proc. Symp., Madison/Wis, 1978.

\bibitem{dahlquist1976relation}
Germud Dahlquist, \emph{On the relation of {G}-stablity to other stability
  concepts for linear multistep methods}, Tech. report, CM-P00069426, 1976.

\bibitem{Dahlquist1978}
Germund Dahlquist, \emph{G-stability is equivalent to {A-stability}}, BIT
  Numerical Mathematics \textbf{18} (1978), no.~4, 384--401.

\bibitem{DLNStabilityofTwoStepMethods}
Germund~G. Dahlquist, Werner Liniger, and Olavi Nevanlinna, \emph{Stability of
  two-step methods for variable integration steps}, SIAM Journal on Numerical
  Analysis \textbf{20} (1983), no.~5, 1071--1085.

\bibitem{discacciati2002SDmodel}
Marco Discacciati, Edie Miglio, and Alfio Quarteroni, \emph{Mathematical and
  numerical models for coupling surface and groundwater flows}, Applied
  Numerical Mathematics \textbf{43} (2002), no.~1-2, 57--74.

\bibitem{ervin2009coupled}
VJ~Ervin, EW~Jenkins, and Shuyu Sun, \emph{Coupled generalized nonlinear
  {S}tokes flow with flow through a porous medium}, SIAM Journal on Numerical
  Analysis \textbf{47} (2009), no.~2, 929--952.

\bibitem{feng2012SDnoniterativeDecomposiion}
Wenqiang Feng, Xiaoming He, Zhu Wang, and Xu~Zhang, \emph{Non-iterative domain
  decomposition methods for a non-stationary {S}tokes--{D}arcy model with
  {B}eavers--{J}oseph interface condition}, Applied Mathematics and Computation
  \textbf{219} (2012), no.~2, 453--463.

\bibitem{girault2009dg}
Vivette Girault and B{\'e}atrice Rivi{\`e}re, \emph{D{G} approximation of
  coupled {N}avier--{S}tokes and {D}arcy equations by
  {B}eaver--{J}oseph--{S}affman interface condition}, SIAM Journal on Numerical
  Analysis \textbf{47} (2009), no.~3, 2052--2089.

\bibitem{guangzhi2017SDlocalFEMBJ}
DU~Guangzhi and ZUO Liyun, \emph{Local and parallel finite element method for
  the mixed {N}avier-{S}tokes/{D}arcy model with {B}eavers-{J}oseph interface
  conditions}, Acta Mathematica Scientia \textbf{37} (2017), no.~5, 1331--1347.

\bibitem{Guzel2018TimeFI}
Ahmet~Baris Guzel and William~J. Layton, \emph{Time filters increase accuracy
  of the fully implicit method}, BIT Numerical Mathematics \textbf{58} (2018),
  301--315.

\bibitem{hairer1993solvingII}
E.~Hairer, S.P. N{\o}rsett, and G.~Wanner, \emph{Solving ordinary differential
  equations {II}: Stiff and differential-algebraic problems}, Solving Ordinary
  Differential Equations, Springer, 1993.

\bibitem{he2015steadySDdomaindecomposition}
Xiaoming He, Jian Li, Yanping Lin, and Ju~Ming, \emph{A domain decomposition
  method for the steady-state {N}avier--{S}tokes--{D}arcy model with
  {B}eavers--{J}oseph interface condition}, SIAM Journal on Scientific
  Computing \textbf{37} (2015), no.~5, S264--S290.

\bibitem{hou2016SDoptimaltwogrid}
Yanren Hou, \emph{Optimal error estimates of a decoupled scheme based on
  two-grid finite element for mixed {S}tokes--{D}arcy model}, Applied
  Mathematics Letters \textbf{57} (2016), 90--96.

\bibitem{hou2019solution}
Yanren Hou and Yi~Qin, \emph{On the solution of coupled {S}tokes/{D}arcy model
  with {B}eavers--{J}oseph interface condition}, Computers \& Mathematics with
  Applications \textbf{77} (2019), no.~1, 50--65.

\bibitem{jiang2019SDefficientEnsemble}
Nan Jiang and Changxin Qiu, \emph{An efficient ensemble algorithm for numerical
  approximation of stochastic {S}tokes--{D}arcy equations}, Computer Methods in
  Applied Mechanics and Engineering \textbf{343} (2019), 249--275.

\bibitem{kanschat2010SDStrongFEM}
Guido Kanschat and B{\'e}atrice Riviere, \emph{A strongly conservative finite
  element method for the coupling of {S}tokes and {D}arcy flow}, Journal of
  Computational Physics \textbf{229} (2010), no.~17, 5933--5943.

\bibitem{Layton:2019:DoublyAdaptiveAC}
William {Layton} and Michael {McLaughlin}, \emph{{Doubly-adaptive artificial
  compression methods for incompressible flow}}, arXiv e-prints (2019),
  arXiv:1907.08235.

\bibitem{2020arXiv200108640L}
William {Layton}, Wenlong {Pei}, Yi~{Qin}, and Catalin {Trenchea},
  \emph{{Analysis of the variable step method of Dahlquist, Liniger and
  Nevanlinna for fluid flow}}, arXiv e-prints (2020), arXiv:2001.08640.

\bibitem{layton2012SDCNLF}
William Layton and Catalin Trenchea, \emph{Stability of two {IMEX} methods,
  {CNLF} and {BDF}2-{AB}2, for uncoupling systems of evolution equations},
  Applied Numerical Mathematics \textbf{62} (2012), no.~2, 112--120.

\bibitem{layton2002coupling}
William~J Layton, Friedhelm Schieweck, and Ivan Yotov, \emph{Coupling fluid
  flow with porous media flow}, SIAM Journal on Numerical Analysis \textbf{40}
  (2002), no.~6, 2195--2218.

\bibitem{li2018SDsecondpartition}
Yi~Li and Yanren Hou, \emph{A second-order partitioned method with different
  subdomain time steps for the evolutionary {S}tokes-{D}arcy system},
  Mathematical Methods in the Applied Sciences \textbf{41} (2018), no.~5,
  2178--2208.

\bibitem{mu2010SDdecoupledmixed}
Mo~Mu and Xiaohong Zhu, \emph{Decoupled schemes for a non-stationary mixed
  {S}tokes-{D}arcy model}, Mathematics of Computation \textbf{79} (2010),
  no.~270, 707--731.

\bibitem{yi2018SDoptimaltwogrid}
Yi~Qin and Yanren Hou, \emph{Optimal error estimates of a decoupled scheme
  based on two-grid finite element for mixed {N}avier-{S}tokes/{D}arcy model},
  Acta Mathematica Scientia \textbf{38} (2018), no.~4, 1361--1369.

\bibitem{qin2019SDtimefilterBE}
\bysame, \emph{The time filter for the non-stationary coupled {S}tokes/{D}arcy
  model}, Applied Numerical Mathematics \textbf{146} (2019), 260--275.

\bibitem{qin2020SDgraddivBE}
Yi~Qin, Yanren Hou, Pengzhan Huang, and Yongshuai Wang, \emph{Numerical
  analysis of two grad--div stabilization methods for the time-dependent
  {S}tokes/{D}arcy model}, Computers \& Mathematics with Applications
  \textbf{79} (2020), no.~3, 817--832.

\bibitem{riviere2005SDdiscontinuousFEM}
B{\'e}atrice Rivi{\'e}re, \emph{Analysis of a discontinuous finite element
  method for the coupled {S}tokes and {D}arcy problems}, Journal of Scientific
  Computing \textbf{22} (2005), no.~1-3, 479--500.

\bibitem{shan2013SDpartitionedBE}
Li~Shan and Haibiao Zheng, \emph{Partitioned time stepping method for fully
  evolutionary {S}tokes--{D}arcy flow with {B}eavers--{J}oseph interface
  conditions}, SIAM Journal on Numerical Analysis \textbf{51} (2013), no.~2,
  813--839.

\bibitem{shan2013SDsubdomain}
Li~Shan, Haibiao Zheng, and William~J Layton, \emph{A decoupling method with
  different subdomain time steps for the nonstationary {S}tokes--{D}arcy
  model}, Numerical Methods for Partial Differential Equations \textbf{29}
  (2013), no.~2, 549--583.

\bibitem{zuo2018SDmultigrid}
Liyun Zuo and Guangzhi Du, \emph{A multi-grid technique for coupling fluid flow
  with porous media flow}, Computers \& Mathematics with Applications
  \textbf{75} (2018), no.~11, 4012--4021.

\bibitem{zuo2018SDparalleltwogrid}
\bysame, \emph{A parallel two-grid linearized method for the coupled
  {N}avier-{S}tokes-{D}arcy problem}, Numerical Algorithms \textbf{77} (2018),
  no.~1, 151--165.

\bibitem{zuo2014SDdecouplingtwogrid}
Liyun Zuo and Yanren Hou, \emph{A decoupling two-grid algorithm for the mixed
  {S}tokes-{D}arcy model with the {B}eavers-{J}oseph interface condition},
  Numerical Methods for Partial Differential Equations \textbf{30} (2014),
  no.~3, 1066--1082.

\bibitem{zuo2015numerical}
\bysame, \emph{Numerical analysis for the mixed {N}avier--{S}tokes and {D}arcy
  problem with the {B}eavers--{J}oseph interface condition}, Numerical Methods
  for Partial Differential Equations \textbf{31} (2015), no.~4, 1009--1030.

\bibitem{zuo2015SDtwogridmixed}
\bysame, \emph{A two-grid decoupling method for the mixed {S}tokes--{D}arcy
  model}, Journal of Computational and Applied Mathematics \textbf{275} (2015),
  139--147.

\end{thebibliography}

\providecommand{\bysame}{\leavevmode\hbox to3em{\hrulefill}\thinspace}
\providecommand{\MR}{\relax\ifhmode\unskip\space\fi MR }
\providecommand{\MRhref}[2]{%
  \href{http://www.ams.org/mathscinet-getitem?mr=#1}{#2}
}
\providecommand{\href}[2]{#2}

\end{document}